\documentclass[11pt,reqno]{amsart}

\usepackage{amsmath,amssymb,mathrsfs}
\usepackage{graphicx,cite,times}
\usepackage{cases}
\usepackage{color}

\newtheorem{theo}{Theorem}[section]

\newtheorem{lemm}[theo]{Lemma}

\numberwithin{equation}{section}

\begin{document}

\title[Multiple Cavities Scattering Problem]{Analysis  of Time-domain  Electromagnetic Scattering Problem   by  Multiple Cavities}

\author{Yang Liu}
\address{School of Mathematics and Statistics, and Center for Mathematics
and Interdisciplinary Sciences, Northeast Normal University, Changchun,
Jilin 130024, P.R.China}
\email{Liuy694@nenu.edu.cn}

\author{Yixian Gao}
\address{School of Mathematics and Statistics, and Center for Mathematics
and Interdisciplinary Sciences, Northeast Normal University, Changchun,
Jilin 130024, P.R.China}
\email{gaoyx643@nenu.edu.cn}

\author{Jian Zu}
\address{School of Mathematics and Statistics, and Center for Mathematics
and Interdisciplinary Sciences, Northeast Normal University, Changchun,
Jilin 130024, P.R.China}
\email{zuj100@nenu.edu.cn}

\thanks{The research of YG was supported in part by NSFC grant 11871140, JJKH20180006KJ, JLSTDP 20190201154JC, 20160520094JH and FRFCU2412019BJ005. The research of JZ was supported in part by
NSFC grand  11571065,11671071.
}

\begin{abstract}
Consider the time-domain multiple cavity scattering problem, which arises in diverse scientific areas and
has significant industrial and military applications.
The  multiple cavity embedded in an infinite ground plane, is filled with inhomogeneous media characterized by variable  dielectric permittivities and magnetic permeabilities.
Corresponding to the transverse  electric or
magnetic  polarization,
the scattering problem can be studied for the Helmholtz equation in frequency domain and
 wave  equation in time-domain, respectively.
 A novel  transparent boundary condition in time-domain  is developed to reformulate
  the cavity scattering problem into an initial-boundary value problem in a bounded domain.
  The well-posedness and stability are  established for the reduced  problem.
  Moreover, a priori estimates for the electric field is obtained with a minimum requirement for the data  by directly
studying the wave equation.

\end{abstract}
\maketitle

\section{Introduction}

This paper is concerned with the mathematical analysis of the time-domain electromagnetic scattering problem of multiple  cavities,
 which is embedded  in a conducting ground planes.
The  cavity scattering problem arises in diverse scientific areas and has significant industrial and military applications,
including the design of cavity-backed conformal antennas for civil and military use, and the characterization of radar cross-section (RCS) of vehicles with grooves, especially to design RCS. It is used to detect airplanes in a wide variation of ranges.
For instance, a stealth aircraft will have design features that give it a low RCS, as opposed to a passenger airliner that will have a high RCS.
RCS is integral to the development of radar stealth technology, particularly in applications involving aircraft and ballistic missiles.
The cavity RCS caused by jet engine inlet ducts or cavity-backed antennas can dominate the total RCS.
A thorough understanding of the electromagnetic scattering characteristic of a target, particularly a cavity,
is necessary for successful implementation of any desired control of its RCS.

The descriptions of cavity scattering problem were centered on methods developed in the time-harmonic   and time-domain.
For the  time-harmonic problems were introduced  firstly by engineers \cite{Jin2002,Jin1998,Jin2003A,Jin1991,wood03}. The
mathematical  analysis of  the cavity scattering problem was given by   three fundamental papers \cite{Ammari2000,AmmariBao2002,AmmariBaoWood2002},
where the existence and uniqueness of the solutions  were obtained based on a non-local transparent boundary condition on the cavity opening.
A large amount of information was available regarding their solutions
for both the two-dimensional Helmholtz and the three-dimensional Maxwell equations\cite{Bao2005An,BaoGao2011,BaoGao2012,Baosun2005,Li2012,LiWuZheng2012,LiandWood2013,wood2017}.
A good survey to the problem of cavity scattering  can be found in \cite{Li2018}.
The time-domain scattering problems have recently attracted considerable
attention due to their capability of capturing wide-band signals and modeling more
general material and nonlinearity\cite{Chen2008,Jin2009Finite,LiandHuang2013,WangZhangZhao2012}, which motivates us to tune our focus
from seeking the best possible conditions for those physical parameters to the time-domain
problem. Comparing with the time-harmonic problems, the time-domain
problems are less studied due to the additional challenge of the temporal dependence.
The analysis can be found in \cite{GaoLi2016,GaoLi2017,GaoLiLi2018,BaoGao2018,GaoLiZhang2017,Wei2019} for the time-domain
acoustic, elastic and electromagnetic scattering problems in different structures including bounded obstacles, periodic surfaces, and unbounded rough surfaces.
 Inspired by the one  open cavity structure  in \cite{LiWood2015}, we extends the results to  the  multiple cavity  scattering problem.
 It appears more  complicated because of the unbounded nature of the domain and
 the novel transparent boundary condition on multiple apertures.
Utilizing the Laplace transform as a bridge between the time-domain and the frequency domain,
we develop an exact time-domain transparent boundary condition (TBC) and reduce the scattering problem equivalently into an initial boundary value problem in a bounded domain.
Using the energy method with new energy functions, we can show the well-posedness and stability of  the time-domain multiple  cavity scattering problem.

The paper is organized as follow.
 In section \ref{FRP1}, we introduce the model problem of one cavity scattering problem
  and establish  a time-domain TBC. Section \ref{TSP} is concentrated on the analysis of two cavities scattering problem, while  the well-posedness and stability  are addressed  in both the frequency and time-domain. The multiple cavity  problem is proposed in section \ref{FRP}, while a priori estimates with explicit time dependence for the quantities of electric  filed is obtained with a minimum requirement for the data  by directly studying the wave equation.
  We conclude the paper with some remarks  in section \ref{CL}.

\section{one cavity scattering problem}\label{FRP1}

In this section, we shall introduce  the mathematical model
for a single cavity scattering problem and develop an exact TBC to
reduce the scattering problem from an unbounded domain into a bounded domain.

\subsection { Problem formulation}
Consider a simpler model for the open cavity scattering problem  by assuming  that the medium and material are invariant   along the $z$-axis.
Let $\Omega \subset \mathbb R^2$ be the cross section of a $z$-invariant cavity with  a Lipschitz continuous boundary $\partial \Omega =S \cup \Gamma$, as seen in Figure \ref{fig1}.
The cavity is filled with some inhomogeneous medium, characterized by the variable  dielectric permittivity $\varepsilon (x, y)$ and magnetic permeability $\mu (x, y)$. The exterior region $\Omega^{\rm e}$  is filled with some homogeneous
material with a constant permittivity $\varepsilon_0$ and a constant permeability $\mu_0$.
Here the
cavity wall $S$ is assumed to be a perfect electric conductor and the cavity opening $ \Gamma$ is aligned
with the perfectly electrically conducting infinite ground surface $\Gamma^{\rm c}$.
An open cavity $\Omega$, enclosed by the aperture $\Gamma$ and the wall $S$, is placed on a perfectly conducting ground plane $\Gamma^{\rm c}$.

\begin{figure}
\centering
\includegraphics[width=0.45\textwidth]{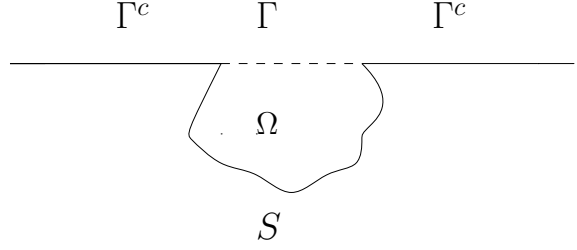}
\caption{The problem geometry of the one cavity.}
\label{fig1}
\end{figure}

The electromagnetic wave propagation is governed by the
 time-domain Maxwell equations
\begin{align}\label{TMW}
\begin{cases}
\nabla \times \boldsymbol {E}(\boldsymbol r, t) +\mu \partial_t \boldsymbol {H}(\boldsymbol r, t)=0,\\
\nabla \times \boldsymbol {H} (\boldsymbol r, t)-\varepsilon \partial_t \boldsymbol
{E}(\boldsymbol r, t)= 0,
 \end{cases}
\end{align}
where $\boldsymbol r= (x, y, z) \in \mathbb R^3,$
$\boldsymbol E$ is the electric field, $\boldsymbol H$ is the magnetic
field, $\varepsilon $ and $\mu$ are the dielectric permittivity and magnetic permeability, respectively,
and  satisfy
\[
0 < \varepsilon_{\rm min} \leq \varepsilon \leq \varepsilon_{\rm max}<
\infty,\quad 0 < \mu_{\rm min} \leq \mu \leq \mu_{\rm max} < \infty,
\]
while  $\varepsilon_{\rm min}, \varepsilon_{\rm max}, \mu_{\rm min}, \mu_{\rm
max}$ are constants.
The system is constrained by the initial conditions
\begin{align}\label{ic1}
 \boldsymbol E \big|_{t=0} =0,\quad \boldsymbol H \big|_{t=0}=0.
\end{align}

Since the structure is invariant  in the $z$-axis, the problem can be decomposed into two fundamental polarizations: transverse electric (TE) and transverse
magnetic (TM). The three-dimensional  Maxwell equations can be reduced to  the two-dimensional wave equation.

(i) TE polarization: the magnetic field is transverse to the $z$-axis, the electric and magnetic fields are
\begin{align*}
 \boldsymbol E (\boldsymbol r, t)= [0, 0 , u (\boldsymbol \rho, t) ]^{\top}, \quad \boldsymbol H (\boldsymbol r, t) = [H_1 (\boldsymbol \rho, t), H_2 (\boldsymbol \rho, t), 0]^{\top},
\end{align*}
where $ \boldsymbol \rho =(x, y) \in \mathbb R^2.$
Eliminating the magnetic field from \eqref{TMW}, we get the wave equation for the electric field
\begin{align}\label{ew}
 \varepsilon \partial_t^2 u - \nabla \cdot \left( \mu^{-1} \nabla u \right)= 0 \quad \text {in }~~ \Omega^{\rm e} \cup \Omega,~~t>0.
\end{align}
By the perfectly conducting boundary condition on the ground plane and cavity wall we can get
\begin{align*}
 u=0 \quad \text {on} ~~~~~S\cup\Gamma^{\rm c}  , ~~~~t>0.
\end{align*}
It follows from  the initial condition \eqref{ic1} that $u (\boldsymbol \rho, t)$ satisfies the homogeneous initial conditions
\begin{align*}
 u (\boldsymbol \rho, t) \big|_{t=0} =0, \quad \partial_t u (\boldsymbol \rho, t) \big|_{t=0}= 0 \quad \text {in}~~\Omega^{\rm e} \cup \Omega.
\end{align*}

(ii) TM polarization: the electric  field is transverse to the $z$-axis, the electric and magnetic fields are
\begin{align*}
 \boldsymbol E (\boldsymbol r, t)= \left[E_1 (\boldsymbol \rho, t), E_2 (\boldsymbol \rho, t), 0 \right]^{\top},
 \quad \boldsymbol H (\boldsymbol r, t) = \left[0, 0, u(\boldsymbol \rho, t) \right]^{\top}.
\end{align*}
We may eliminate the electric field from \eqref{TMW} and obtain the wave equation for the magnetic field
\begin{align}\label{hw}
 \mu \partial_t^2 u -\nabla \cdot \left( \varepsilon^{-1}\nabla u\right)=0 \quad \text {in}~~~~\Omega^{\rm e} \cup \Omega, ~~~~t>0.
\end{align}
It also follows from the  perfectly conducting boundary condition on the ground plane and cavity wall that
\begin{align*}
 \partial_{\boldsymbol \nu} u =0 \quad  \text {on}~~S\cup \Gamma^{\rm c} , ~~~~t>0,
\end{align*}
where $\boldsymbol \nu$ is the unit outward normal vector on $S\cup \Gamma^{\rm c}.$   The initial conditions  for the TM  is
\begin{align*}
 u (\boldsymbol \rho, t) \big|_{t=0}=0, \quad  \partial_t u (\boldsymbol \rho, t) \big|_{t=0}=0  \quad \text {in}~~~\Omega^e\cup\Omega.
\end{align*}

It is clear to note from \eqref{ew} and \eqref{hw} that  TE and TM polarizations can be handled in a
unified way by formally exchanging the roles of $\varepsilon$ and $\mu$. We will just discuss the results in detail by using \eqref{ew} (TE case)
as the model equation in the rest of the paper. The method can be extended to the TM polarization
with obvious modifications.

Let an incoming plane wave $u^{\rm inc}=f(-t-c_1 x-c_2 y)$ be incident on the cavity from above, where  $ f$ is a smooth function and its regularity will be specified later,  and $c_1=\cos \theta/\sqrt{\varepsilon_0 \mu_0}, c_2= \sin \theta/\sqrt{\varepsilon_0 \mu_0}, 0<\theta<\pi.$  Clearly, the incident field satisfies the wave equation \eqref{ew} with $\varepsilon =\varepsilon_0, \mu =\mu_0.$ The total field can be split into
the incident field, the reflected field and the scattered field:
\begin{align*}
u=u^{\rm inc} +u^{\rm r}+u^{\rm sc},
\end{align*}
where  $u^{\rm r}=-f (-t-c_1 x +c_2 y)$ (or $u^{\rm r}=f (-t-c_1 x +c_2 y)$ ) is the reflected field in TE (or TM) case .
To impose
the initial conditions, we assume that the total field, the incident field and the reflected field  vanish for $t<0$, so that the
scattered field $u^{\rm sc}=0$ for $t<0$.
 Moreover, the   scattered field is required to  satisfies the Sommerfeld radiation condition:
\begin{align}\label{SR}
 \frac{1}{\sqrt{\varepsilon_0 \mu_0}}\partial_r u^{\rm sc} +  \partial_t u^{\rm sc} = o(r^{-1/2}) \quad \text {as}~~r= |\boldsymbol \rho| \rightarrow \infty,~~~t>0.
\end{align}

To analyze the problem, the open domain needs to be truncated into a bounded domain. Therefore, a suitable boundary condition has to be imposed on the
boundary of the bounded domain so that no artificial wave reflection occurs.
We shall present a transparent boundary condition on the open domain enclosing the inhomogeneous
cavity.

\subsubsection{ Laplace transform  and some notation}
We first introduce the Laplace transform  and present some identities  for the transform.
For any $s=s_1+{\rm i}s_2$ with $s_1,
s_2\in\mathbb{R},~s_1>0, ~{\rm i} =\sqrt{-1}$, define by $\breve{u}(s)$ the Laplace transform of the
function $u(t)$, i.e.,
\[
 \breve{u}(s)=\mathscr{L}(u)(s)=\int_0^\infty e^{-st}u(t){\rm d}t.
\]
Using the integration by parts yields
\begin{align*}
\int_0^t  u (\tau) {\rm d} \tau =\mathscr L^{-1} (s^{-1} \breve { u} (s)),
\end{align*}
where $\mathscr L^{-1}$ is the inverse Laplace transform.
One verify form the formula of the inverse Laplace transform that
\begin{align*}
 u(t)=\mathscr F^{-1} \left( e^{s_1 t} \mathscr L ( u) (s_1+ {\rm i} s_2)\right),
\end{align*}
where $\mathscr F ^{-1}$ denotes the inverse Fourier transform with respect to $s_2.$

Recalling  the
Plancherel or Parseval identity for the Laplace transform (cf.
\cite[(2.46)]{02})
 \begin{equation}\label{PI}
 \frac{1}{2 \pi} \int_{- \infty}^{\infty} \breve {u} (s)  \breve {v} (s) {\rm d} s_2= \int_0^{\infty} e^{- 2 s_1 t}
  { u} (t)
 { v} (t) {\rm d} t,\quad\forall ~ s_1>\sigma_0>0,
 \end{equation}
where $\breve  { u}= \mathscr L ( u), \breve { v}= \mathscr L ( v)$ and $\sigma_0$ is
abscissa of convergence for the Laplace transform of $ u$ and $ v.$

Hereafter, the expression $a \lesssim  b$ stands for $a \leq  C b$, where $C$ is a positive
constant and its specific value is not required but should be always clear from the context.

The following lemma (cf.\cite[Theorem 43.1]{04})  is an
analogue of Paley--Wiener--Schwarz theorem for Fourier transform of the
distributions with compact support in the case of Laplace transform.
\begin{lemm}\label {A2}
 Let $\breve {\boldsymbol h} (s)$ denote a holomorphic function in the half-plane
$s_1 > \sigma_0$ , valued in the Banach space $\mathbb E$. The two following conditions are
equivalent
\begin{enumerate}

\item there is a distribution $ \breve {\boldsymbol h} \in \mathcal D_{+}'(\mathbb E)$ whose Laplace
transform is equal to $\breve h(s)$,

\item there is a real $\sigma_1$ with $\sigma_0 \leq \sigma_1 <\infty$ and an
integer $m \geq 0$ such that for all complex numbers $s$ with ${\rm Re} ~s =s_1 >
\sigma_1,$ it holds that $\| \breve {\boldsymbol  h} (s)\|_{\mathbb E} \lesssim (1+|s|)^{m}$,

\end{enumerate}
where $\mathcal D'_{+}(\mathbb E)$ is the space of distributions on the real
line which vanish identically in the open negative half line.
\end{lemm}

Next, we introduce some function space notation.
Let  $\Omega \subset \mathbb R^2$ be a bounded Lipschitz domain with boundary $\partial \Omega.$
Denote the Sobolev space: $H^{1}(\Omega)=\left\{ u: {\rm D}^{\alpha} u  \in L^2 (\Omega) \text{~~for~all} ~~|\alpha| \leq 1\right\}$.
To describe the boundary operator and transparent boundary condition in the formulation of the boundary value problem, we define the trace functional space
\begin{align*}
H^{\nu}(\mathbb R)=\left\{u\in L^{2}(\mathbb R): \int_{\mathbb R}(1+\xi^{2})^{\nu}|\hat u|^{2}{\rm d}\xi<\infty \right\},
\end{align*}
whose norm is defined by
\begin{align*}
\|u\|_{H^{\nu}(\mathbb R)}=\left( \int_{\mathbb R}(1+\xi^{2})^{\nu}|\hat u|^{2}{\rm d}\xi\right)^{1/2},
\end{align*}
where $\hat u$ is the Fourier transform of $u$ defined as
\begin{align*}
\hat u (\xi) = \int_{\mathbb R} u(x) e^{{\rm i} x \xi} {\rm d} x.
\end{align*}
It is clear to note that the
dual space of $H^{1/2} (\mathbb R)$ is $H^{-1/2} (\mathbb R) $ under the $L^2 (\mathbb R)$ inner produce
\[
 \langle u, v \rangle =\int_{\mathbb R} u \bar v {\rm d} x= \int_{\mathbb R} \hat u \bar {\hat v} {\rm d} \xi.
\]

\subsubsection{Transparent  boundary condition}
We introduce a time-domain TBC to formulate the cavity scattering problem into an equivalent initial-boundary value problem in a bounded domain.
The idea is to design a Dirichlet--to--Neumann (DtN) operator
which maps the Dirichlet data to the Neumann data of the wave field.
More precisely, we will address  the reduced initial-boundary value problem
\begin{align}\label{rdp1}
\begin{cases}
  \varepsilon \partial_t^2 u - \nabla \cdot \left(  \mu^{-1} \nabla u\right)=0 \quad  &\text {in} ~~~\Omega,~~~t>0,\\
  u \big|_{t=0}=0, \quad \partial_t u \big|_{t=0} =0\quad & \text {in}~~~\Omega,\\
  u=0  \quad & \text {on}~~~S, ~~~~t>0,\\
  \partial_{\boldsymbol n} u = \mathscr T ~u+g \quad & \text {on}~~~\Gamma,~~~~t>0,
  \end{cases}
\end{align}
where $\mathscr T$ is a time-domain boundary operator  and  $g$  will be given later.
In what follows, we derive the formulation of the operator $\mathscr T $ and analyze its important properties.

Since $\varepsilon=\varepsilon_0, \mu=\mu_0$ in $\Omega^{\rm e},$  the  equations  \eqref{ew} and \eqref{hw} together with the radiation condition \eqref{SR}  implies the scattered field $u^{\rm sc} $ satisfies
\begin{align}\label{crp}
\begin{cases}
 \Delta u^ {\rm sc }-\varepsilon_0 \mu_0\partial_t^2 u^{\rm sc}  =0 \quad &\text {in}~~~\Omega^{\rm e},~~~t>0,\\
 u^{\rm sc}\big|_{t=0}=0, \quad \partial_t u^{\rm sc }\big|_{t=0}=0 \quad &\text {in}~~~~~\Omega^{\rm e},\\
 u^{\rm sc} =-\left( u^{\rm inc} + u^{\rm r} \right)  \quad  &\text {on}~~~ \Gamma^{\rm c},~~~t>0,\\
{(\varepsilon_0 \mu_0)^{-1/2}}\partial_t u^{\rm s} + \partial_r u^{\rm s} = o(r^{-1/2}) \quad &\text {as}~~r=|\boldsymbol \rho| \rightarrow \infty,~~~t>0.
 \end{cases}
\end{align}
Let $\breve u (\boldsymbol \rho, s) = \mathscr L (u) (\boldsymbol \rho, t)$ be the Laplace transform of $u (\boldsymbol \rho, t)$ with respect to $t$. Recalling that
\begin{align*}
 \mathscr L (\partial_t u) = s \breve u (\cdot, s )- u (\cdot, 0),\quad
 \mathscr  L (\partial_t^2 u)= s^2 \breve {u} (\cdot, s)- s u (\cdot, 0)-\partial_t u (\cdot, 0).
\end{align*}
Taking the Laplace transform of \eqref{crp} with the initial conditions, we can get  the time-harmonic Helmhlotz equation for the scattered field  with the complex wave number
\begin{align}\label{he}
\begin{cases}
 \Delta \breve u^{\rm sc} -  \frac{ s^2} {c^2} \breve u^{\rm sc}=0 \quad &\text {in}~\Omega^{\rm e},\\
 \breve u^{\rm sc} =- \left( \breve{u}^{\rm inc} + \breve {u}^{\rm r} \right)  \quad  &\text {on}~~~ \Gamma^{\rm c},\\
 \frac{ s} {c}\breve u^{\rm sc} + \partial_r \breve  u^{\rm sc} = o(r^{-1/2}) \quad &\text {as}~~r=|\boldsymbol \rho| \rightarrow \infty,
 \end{cases}
\end{align}
where $c :=\frac{1}{\sqrt{\varepsilon_0 \mu_0}}$ is the light
speed in the free space.

By taking the Fourier transform of the first equation in \eqref{he} with respect to $x$, we have an ordinary differential equation with respect to
$y$:
\begin{align}\label{pc}
 \frac{\partial^{2}\hat {\breve{u}}^{\rm sc}}{\partial y^{2}}-\left(\xi^{2}+\frac{s^{2}}{c^{2}} \right)\hat {\breve{u}}^{\rm sc}=0, \quad y>0.
\end{align}
It follows form the radiation condition in \eqref{he}, we deduce that the solution of \eqref{pc} has the analytical form
\begin{align}\label{hc}
\hat{\breve{u}}^{\rm sc}=\hat{\breve{u}}^{\rm sc}( \xi, 0)e^{\beta(\xi)y},
\end{align}
where
\begin{align}\label{bdd}
\beta(\xi)=\left(\xi^{2}+\frac{s^{2}}{c^{2}}\right)^{1/2} \quad  \text{ with}  ~~{\rm Re}\left(\beta(\xi)\right)<0.
\end{align}
Taking the inverse Fourier transform of \eqref{hc}, we find that
\begin{align*}
\breve u^{\rm sc}(x, y)=\int_{\mathbb R} \hat{\breve{u}}^{\rm sc}( \xi, 0)e^{\beta(\xi)y} e^{-{\rm i} x \xi} {\rm d} \xi \quad \text {in}~\Omega^{\rm e}.
\end{align*}
Taking the normal derivative on $\Gamma^{\rm c}\cup\Gamma$ and evaluating at $y=0$ yields
\begin{align}\label{abc}
\partial_{\boldsymbol n}\breve{u}^{\rm sc} (x, y)\big|_{y=0}=\int_{\mathbb R}  \beta (\xi)\hat{\breve{u}}^{\rm s}( \xi, 0) e^{-{\rm i} x \xi} {\rm d} \xi ,
\end{align}
where $\boldsymbol n$ is the unit outward normal on $\Gamma^{\rm c} \cup \Gamma$, i.e. $\boldsymbol n= (0, 1).$

For any $w \in H^{1/2} (\Gamma^{\rm c} \cup \Gamma)$ with $w= \int_{\mathbb R} \hat w (\xi, 0) e^{{-\rm i} x \xi } {\rm d} \xi$,  define the boundary operator
$\mathscr B $
\begin{align}\label{DF}
\mathscr B w:= \int_{\mathbb R} \beta (\xi) \hat w (\xi,  0) e^{{-\rm i} x \xi } {\rm d} \xi,
\end{align}
which leads to a transparent boundary condition for the scattered field on $\Gamma^{\rm c} \cup \Gamma$:
\begin{align*}
\partial_{\boldsymbol n}  \breve u^{\rm sc}= \mathscr B \breve u^{\rm sc}.
\end{align*}
From $ \breve u^{\rm sc}=\breve u-\left( \breve  u^{\rm inc} + \breve  u^{\rm r} \right)$, we can get an equivalent transparent boundary condition for the
total field
\begin{align}\label{hn}
\partial_{\boldsymbol n }\breve{u}=\mathscr B\breve{u}+\breve{g} \quad  \text {on}~~\Gamma^{\rm c} \cup \Gamma,
\end{align}
where $\breve{g}=\partial_{\boldsymbol n} \left( \breve{u}^{\rm inc} +\breve u^{\rm r} \right) -\mathscr B \left( \breve u^{\rm inc} + \breve{u}^{r} \right)$.

Taking the inverse Laplace transform of \eqref{hn} yields the TBC in the time-domain
\begin{align}\label{hm}
\partial_{\boldsymbol n}u=\mathscr T{u}+g \quad & \text {on}~~~\Gamma^{\rm c} \cup \Gamma,~~~~t>0,
\end{align}
where $\mathscr T=\mathscr L^{-1} \circ \mathscr B\circ \mathscr L$ and $g=\mathscr L^{-1}\circ\breve{g}$.

Since $u$ is defined on $\Gamma^{\rm c } \cup \Gamma$ and the  transparent boundary condition above  is derived for $u \in H^{1/2} (\mathbb R)$.
In order to  derive the transparent boundary condition for the total field on $\Gamma $,   we   make the zero extension as follows:
 for any given $u$ on $\Gamma$,
define
\begin{align*}
\tilde u (x)
=\begin{cases}
u \quad & \text{for}~~ x \in \Gamma,\\
0 \quad &\text{for}~~ x \in \Gamma^{\rm c}.
\end{cases}
\end{align*}
Since the  cavity is placed on a perfectly conducting ground
plane $\Gamma^{\rm c}$, i.e. the total filed  is required to be zero on $\Gamma^{\rm c}$, it is obviously that above  zero extension is consistent with the problem geometry.
 Based on the extension and the transparent boundary conditions \eqref{hn} and \eqref{hm}, we have the transparent boundary conditions for the total field on the opening
\begin{align*}
\partial_{\boldsymbol n} \breve u = \mathscr B  \breve  {\tilde u}+\breve g \quad \text {on}~~\Gamma
;
\quad \partial_{\boldsymbol n} u =\mathscr T \tilde u +g \quad \text{on}~~\Gamma, ~t>0.
\end{align*}
Define a dual paring $\langle \cdot, \cdot \rangle_{\Gamma}$ by
\begin{align*}
\langle u, v\rangle_{\Gamma}=\int_{\Gamma} u \bar v {\rm d} \gamma.
\end{align*}
By the definition of extension,  this dual paring for $u$ and $v$ is equivalent to the scalar product in $L^2 (\mathbb R)$
for their extension, i.e.,
\begin{align*}
\langle u, v \rangle_{\Gamma}=\langle \tilde u, \tilde v \rangle.
\end{align*}

The following lemmas are useful in the proof of the well-posedness of the reduced problem.

\begin{lemm}\label{ct1}
 The boundary operator $\mathscr B :H^{1/2} (\mathbb R)\rightarrow H^{-1/2} (\mathbb R)$ is continuous, i.e.,
 \begin{align*}
  \|\mathscr B u\|_{H^{-1/2}(\mathbb R)} \leq C \|u\|_{H^{1/2}(\mathbb R)}, \quad \forall u \in H^{1/2}(\mathbb R).
 \end{align*}

\end{lemm}
\begin{proof}
 For any $u,v \in H^{1/2}(\mathbb R)$, it follows from the definitions \eqref{DF} that
 \begin{align*}
\langle \mathscr B u, v\rangle=\int_{\mathbb R}\mathscr B u~\bar{v} {\rm } d\xi
=\int_{\mathbb R}\frac{\beta(\xi)}{(1+\xi^{2})^{1/2}}(1+\xi^{2})^{1/4}u\cdot(1+\xi^{2})^{1/4}\bar{v}{\rm d}\xi.
 \end{align*}
 To prove the lemma, it is required to estimate
\begin{align*}
\frac{|\beta(\xi)|}{|(1+\xi^{2})|^{1/2}}, \quad -\infty<\xi<\infty.
\end{align*}
Let \[
\frac{s^2}{c^2}=a+{\rm i} b,\quad a:= \frac{s_1^2 -s_2^2}{c^2}, b:=  \frac{2 s_1 s_2}{c^2}.
\]
 Denote
\[
\beta ^2 (\xi) = \frac{s^2}{c^2} +\xi^2=\phi + {\rm i} b,
\]
where $\phi:= {\rm Re} \left( \frac{s^2}{c^2} \right) + \xi^2 = a+ \xi^2. $    A simple calculation gives
\begin{align*}
\frac{|\beta(\xi)|}{|(1+\xi^{2})|^{1/2}}=\left[ \frac{\phi^2 +b^2}{(1+\phi -a)^2} \right]^{1/4}.
\end{align*}
Define
\begin{align*}
F(t)= \frac{t^2 +b^2}{(1+t -a)^2},  \quad  t\geq a.
\end{align*}
It follows
\begin{align*}
F'(t)= \frac{2 (1+t -a) (t (1-a) -b^2)}{(1+t-a)^4}.
\end{align*}
We consider it in two cases:\\
(i) $1-a >0$.   It can be verified that the function  $F(t)$  decreases for $a \leq t \leq  \frac{b^2}{1-a}$  and increase for $t >  \frac{b^2}{1-a}.$
Thus
\[
F(\phi) \leq  \max \left\{ F(a)=a^2+b^2,  ~~ F(+\infty)=1 \right\}.
\]
(ii) $1-a \leq 0.$  It is easy to verify that $F(t)$ decreases for $t \geq a.$  Thus, we have
\begin{align*}
F(\phi) \leq F(a)=a^2 +b^2.
\end{align*}
Combining above estimates and using the Cauchy--Schwarz inequality yield
\begin{align*}
|\langle \mathscr B u, v\rangle|\leq C\|u\|_{H^{1/2}(\mathbb R)}\|v\|_{H^{1/2}(\mathbb R)},
 \end{align*}
where
\begin{align*}
C=\max\{(a^2+b^2)^{1/4}, 1\}.
\end{align*}
Thus we have
\begin{align*}
 \|\mathscr B u\|_{H^{-1/2}(\mathbb R)} \leq\sup\limits_{v \in H^{1/2}(\mathbb R)}\frac{|\langle \mathscr B u, v\rangle|}{\|v\|_{H^{1/2}(\mathbb{R})}}\leq C \|u\|_{H^{1/2}(\mathbb R)}.
\end{align*}
\end{proof}
It follows from Lemma \ref{A2} and Lemma \ref{ct1} that the inverse Laplace transform in \eqref{hm} is  make sense.
\begin{lemm}\label{op1}
 It holds that
 \begin{align*}
  -{\rm Re} \langle (s\mu)^{-1} \mathscr B u, u \rangle \geq 0, \quad u \in H^{1/2} (\mathbb R).
 \end{align*}

\end{lemm}
\begin{proof}
By  the definition \eqref{DF}, we find
\begin{align*}
-\langle(s\mu)^{-1} \mathscr B u, u \rangle=-\int_{\mathbb R}(s\mu)^{-1} \beta(\xi)|u|^{2}{\rm d}\xi =-\int_{\mathbb R} \frac{\bar s  \beta (\xi)}
{\mu |s|^2}|u|^{2}{\rm d}\xi.
\end{align*}
Let $\beta(\xi)=\varsigma+{\rm i}\varrho $ with $\varsigma<0$. Taking the real part of the above equation gives
\begin{align}\label{hk}
 - {\rm Re} \langle (s\mu)^{-1}  \mathscr B u, u \rangle
 =-\int_{\mathbb R}\frac{s_{1}\varsigma+s_{2}\varrho}{\mu|s|^{2}}|u|^{2}{\rm d}\xi.
\end{align}
Recalling $\beta^{2}(\xi)=\xi^{2}+c^{-2}s^{2}$, we have
\begin{align}\label{ht}
 \varsigma^{2}-\varrho^{2}=\xi^{2}+c^{-2}(s^{2}_{1}-s^{2}_{2}), \quad
 \varsigma \varrho=c^{-2}s_{1}s_{2}.
\end{align}
Using \eqref{ht}, it gives
\begin{align}\label{hp}
s_{1}\varsigma +s_{2}\varrho=\frac{s_{1}}{\varsigma}\left( \varsigma^{2}+c^{-2}s^{2}_{2}\right).
\end{align}
Substituting \eqref{hp} into \eqref{hk}, we have
\begin{align*}
- {\rm Re} \langle (s\mu)^{-1}  \mathscr B u, u \rangle
 =-\int_{\mathbb R}\frac{1}{\mu|s|^{2}} \frac{s_{1}}{\varsigma }\left( \varsigma^{2}+c^{-2}s^{2}_{2}\right) |u|^{2}{\rm d}\xi\geq0, \end{align*}
which completes the proof.
\end{proof}

\begin{lemm}\label{op2}
 For any $u(\cdot, t) \in L^2 \left(0, T; H^{1/2} ( \Gamma) \right)$ with initial value $u (\cdot, 0)=0,$ it holds that
 \begin{align*}
  - {\rm Re}\int_0^T \langle \mathscr T u (\cdot, t), \partial_t u (\cdot, t) \rangle_{\Gamma}  {\rm d} t \geq 0.
 \end{align*}

\end{lemm}
\begin{proof}
 Let $\tilde u (\cdot, t)$ be the extension of $u (\cdot, t)$ with respect to $t$ in $\mathbb R$ such that
 $\tilde u (\cdot, t)= 0$ outside the interval $[0, T]$, and $\breve {\tilde u} =\mathscr L (\tilde u)$ be the Laplace of $\tilde u$. By the Parseval identity \eqref{PI} and Lemma \ref{op1}, we get
 \begin{align*}
  -{\rm Re} \int_0^{T} e^{- 2 s_1 t} \langle \mathscr T u, \partial_t u \rangle_{\Gamma} {\rm d} t
  =&-{\rm Re}\int_0^T e^{- 2 s_1 t} \int_{\Gamma} (\mathscr T u) \partial_t \bar u {\rm d} \gamma {\rm d }t \\
  =& - {\rm Re} \int_{\Gamma} \int_0^{\infty} e^{- 2 s_1 t} (\mathscr T \tilde u) \partial_t \bar {\tilde u} {\rm d} t{\rm d} \gamma\\
  =&-\frac{1}{2 \pi} \int_{- \infty}^{\infty}  {\rm Re} \langle \mathscr B \breve {\tilde u}, s \breve{\tilde u} \rangle_{\Gamma} {\rm d} s_2\\
  =& -\frac {1}{2 \pi} \int_{- \infty}^{\infty} |s|^2\mu {\rm Re} \langle (s\mu)^{-1} \mathscr B \breve {\tilde u}, \breve {\tilde u} \rangle_{\Gamma} {\rm d} s_2  \geq 0,
 \end{align*}
which  completes the proof
after taking $s_1 \rightarrow 0 $.
\end{proof}
The following trace theorem are useful in the following reduced problem, the proof   can be found in (cf.
\cite{03}).
\begin{lemm}\label{tt10}
{\rm (trace theorem)}
 Let $\Omega \subset \mathbb R^2$ be a bounded Lipschitz domain with boundary $\partial \Omega.$ For $1/2 < \nu< 3/2$ the interior trace
operator
\[
 T_0: H^{\nu} (\Omega) \rightarrow H^{\nu-1/2} (\Gamma)~~ \text {is  ~~ bounded}, \quad  \forall w \in H^{\nu} (\Omega),
\]
where $T_0  w =  w  \big|_{\Gamma}$.
\end{lemm}

\subsection{The reduced one cavity scattering problem}

In this section, we will present the well-posedness of the reduced problem  by a variation method,  and given the  stability of
one  cavity scattering problem.

\subsubsection{well-posedness in the $s$-domain}
Taking the Laplace transform of \eqref{rdp1} and using the transparent boundary condition, we may consider the following reduced boundary value problems
\begin{align}\label{au1}
\begin{cases}
  s \varepsilon  \breve{u}-\nabla \cdot \left( (s\mu)^{-1} \nabla \breve{u} \right) = 0  \quad &\text {in}~~~~\Omega,\\
  \breve u=0  \quad &\text {on}~~~~S,\\
  \partial_r \breve{u} =\mathscr B \breve{u}+\breve{g}  \quad &\text {on}~~~~\Gamma,
 \end{cases}
\end{align}
where $\breve g= \partial_{\boldsymbol n} \left( \breve{u}^{\rm inc} +\breve u^{\rm r} \right) -\mathscr B \left( \breve u^{\rm inc} + \breve{u}^{r} \right)$,  $s= s_1+{\rm i} s_2$ with $s_1>0$.

By multiplying a test function $v \in H^1_{\rm S}:=\left\{ v \in H^1(\Omega): v=0 ~~ \text {on}~~S \right\}$ and integrating  by parts,
we arrive at the variational formulation of \eqref{au1}: find $\breve{u} \in H^1_{\rm S} (\Omega)$ such that
\begin{align}\label{vpe}
 a_{1} (\breve{u}, v) =\langle\breve{g},v\rangle_{\Gamma}, \quad \forall  v \in H^1_{\rm S} (\Omega),
\end{align}
where the sesquilinear form
\begin{align}\label{sf1}
 a_{1} (\breve{u}, v) =\int_{\Omega} \left(  (s \mu)^{-1} \nabla \breve{u} \cdot \nabla \bar {v} + s \varepsilon \breve{u} \bar  v \right){\rm d} \boldsymbol \rho
 - \langle (s\mu)^{-1} \mathscr B\breve{u} , v \rangle_{\Gamma}.
\end{align}
\begin{theo}\label{at1}
 The variational problem \eqref{vpe} has a unique solution $\breve{u} \in H^1_{\rm S} (\Omega)$ which satisfies
 \begin{align}\label{es1}
  \|\nabla \breve{u}\|_{L^2 (\Omega)^2} +\|s \breve{u}\|_{L^2 (\Omega)} \lesssim s_1^{-1} \| s\breve{g}\|_{H^{-1/2}(\Gamma)}.
 \end{align}
\end{theo}
\begin{proof}
 It suffices to show the coercivity of the sesquilinear form of $a_{1} (\breve{u}, v)$. The continuity of sesquilinear form  follows directly from the Cauchy--Schwarz inequality, Lemma \ref{ct1} and Lemma \ref{tt10}
 \begin{align*}
  |a_{1} (\breve{u}, v)| \leq
  &\frac{1}{|s| \mu_{\min}} \|\nabla \breve{u}\|_{L^2(\Omega)^2} \|\nabla v\|_{L^2 (\Omega)^2}
  +|s|\varepsilon_{\max} \|\breve{u}\|_{L^2 (\Omega)}\|v\|_{L^2 (\Omega)}\\
  &+\frac{1}{|s| \mu_{\min}}\|\mathscr B \breve{u}\|_{H^{-1/2} (\Gamma)} \|v\|_{H^{1/2}(\Gamma)}\\
  \lesssim & \|\breve{u}\|_{H^1(\Omega)} \|v\|_{H^1(\Omega)}.
 \end{align*}
Letting $v=\breve{u}$ in \eqref{sf1}, we get
\begin{align}\label{sqp}
 a_{1} (\breve{u}, \breve{u}) =\int_{\Omega} \left( (s \mu)^{-1} |\nabla \breve{u}|^2 + s \varepsilon |\breve{u}|^2 \right) {\rm d} \boldsymbol \rho-\langle (s\mu)^{-1 } \mathscr B \breve{u}, \breve{u} \rangle_{\Gamma}.
\end{align}
Taking the real part of \eqref{sqp} and using Lemma \ref{op1}, yields
\begin{align}\label{rp1}
 {\rm Re} \left(  a_{1}(\breve{u}, \breve{u}) \right) \geq C_{1} \frac{s_1}{|s|^2} \left( \|\nabla \breve{u}\|^2_{L^2 (\Omega)^2}  +\|s \breve{u}\|^2_{L^2 (\Omega)}\right),
\end{align}
where $C_{1}=\min\{\mu^{-1}_{\max},1\}$.

It follows from the Lax--Milgram lemma that the variational problem \eqref{vpe} has a unique solution $\breve{u} \in H^1_{\rm S} (\Omega).$
Moreover, we have from \eqref{vpe} that
\begin{align}\label{rp2}
 |a_{1}(\breve{u}, \breve{u})| \leq |s|^{-1} \|\breve{g}\|_{H^{-1/2}(\Gamma)} \|s \breve{u}\|_{L^2 (\Omega)}.
\end{align}
Combining \eqref{rp1}--\eqref{rp2} leads to
\begin{align*}
 \|\nabla \breve{u}\|^2_{L^2 (\Omega)^2}  +\|s \breve{u}\|^2_{L^2 (\Omega)} \lesssim s_1^{-1}  \| s \breve{g}\|_{H^{-1/2}(\Gamma)} \|s \breve{u}\|_{L^2 (\Omega)},
\end{align*}
which gives  estimate of \eqref{es1} after applying the Cauchy--Schwarz inequality.

\end{proof}

\subsubsection{well-posedness in the time-domain}
Using the time-domain transparent boundary condition, we consider the reduced initial-boundary value problem
\begin{align}\label{au2}
\begin{cases}
  \varepsilon \partial_t^2 u - \nabla \cdot \left(  \mu^{-1} \nabla u\right)=0 \quad  &\text {in} ~~~\Omega,~~~t>0,\\
   u\big|_{t=0}=0, \quad \partial_t u \big|_{t=0} =0\quad & \text {in}~~~\Omega,\\
  u=0  \quad & \text {on}~~~S, ~~~~t>0,\\
  \partial_r u = \mathscr T u+g \quad & \text {on}~~~\Gamma,~~~~t>0.
\end{cases}
\end{align}

\begin{theo}\label{ES2} The initial-boundary problem \eqref{au2} has a unique solution $u$, which satisfies
\begin{align*}
u\in L^{2}(0, T; H^{1}_{\rm S}(\Omega))\cap H^{1}(0, T; L^{2}(\Omega)),
\end{align*}
and the stability estimate
\begin{align} \label{UU1}
 \begin{split}
 \max\limits_{t\in[t,T]} &\left(\|\partial_t u\|_{L^2 (\Omega)}  +\|\partial_{t}(\nabla u)\|_{L^2 (\Omega)^2}\right) \\
 & \lesssim  \|g\|_{L^{1}(0, T; H^{-1/2}(\Gamma))}+ \max\limits_{t\in[t,T]}\|\partial_{t}g\|_{H^{-1/2}(\Gamma)}+\|\partial^{2}_{t}g\|_{L^{1}(0, T; H^{-1/2}(\Gamma))}.
 \end{split}
\end{align}
\end{theo}
\begin{proof}
First, we have
\begin{align*}
\int_0^T
\left( \|\nabla {u}\|_{L^2(\Omega)^2}^2+\|\partial_t
{u}\|^2_{L^2(\Omega)}\right) {\rm d} t
&\leq  \int_0^T e^{- 2 s_1(t- T)}\left ( \|\nabla {u}\|_{L^2(\Omega)^2}^2
+\|\partial_t {u}\|^2_{L^2(\Omega)}\right) {\rm d} t\\
&= e^{2 s_1 T} \int_0^T e^{-2 s_1 t} \left ( \|\nabla {u}\|_{L^2(\Omega)^2}^2
+\|\partial_t {u}\|^2_{L^2(\Omega)}\right) {\rm d} t\\
&\lesssim  \int_0^{\infty} e^{-2 s_1 t}  \left ( \| \nabla{u}\|_{L^2(\Omega)^2}^2
+\|\partial_t {u}\|^2_{L^2(\Omega)}\right) {\rm d} t.
\end{align*}
Hence it suffices to estimate the integral
\begin{align*}
\int_0^{\infty} e^{-2 s_1 t}  \left( \|\nabla {u}\|_{L^2(\Omega)^2}^2
+\|\partial_t {u}\|^2_{L^2(\Omega)}\right) {\rm d} t.
\end{align*}
Let $\breve u = \mathscr L u.$
By Theorem  \ref{at1},  we have
\begin{align*}
  \|\nabla \breve u\|^{2}_{L^2 (\Omega)^2} +\|s \breve u\|^{2}_{L^2 (\Omega)} \lesssim  s_1^{-2} |s|^{2}\|\breve{g}\|^{2}_{H^{-1/2}(\Gamma)}\lesssim  s_1^{-2} |s|^{2}\|\breve{u}^{\rm inc}+\breve u^{\rm r}\|^{2}_{H^{1}(\Omega)}.
 \end{align*}
It follows from (cf.\cite[Lemma 44.1]{04}) that $\breve u$ is a holomorphic
function of $s$ on the half plane $s_1 >\sigma_0>0,$  where $\sigma_0$ is
any positive constant. Hence we have from Lemma \ref{A2} that the inverse
Laplace transform of $\breve u$ exists and is supported in $(0, \infty).$

One may verify from the inverse Laplace transform that
\begin{align*}
\breve{u}=\mathscr L {(u)} =\mathscr F (e^{-s_1 t} u),
\end{align*}
where $\mathscr F$ is the Fourier transform in $s_2$.
Recalling  the Plancherel or Parseval identity for the Laplace transform  in \eqref{PI},
it follows
\begin{align*}
\int_0^{\infty} e^{-2 s_1 t}
 &\left( \|\nabla u\|_{L^2(\Omega)^2}^2
+\|\partial_t u\|^2_{L^2(\Omega)}\right) {\rm d} t
 =\frac{1}{2 \pi} \int_{-\infty}^{\infty} \left( \|\nabla  \breve u\|_{L^2(\Omega)^2}^2
+\|s  \breve u\|^2_{L^2(\Omega)}\right) {\rm d} s_2\\
& \lesssim s_1^{-2} \int_{-\infty}^{\infty} \left( \|s (\breve u^{\rm inc} + \breve u^{\rm r})\|^2_{L^2 (\Omega)}  +\| s (\nabla \breve u^{\rm inc} +\nabla \breve  u^{\rm r})\|^2_{L^{2} (\Omega)^{2}}\right) {\rm d} s_2 .
\end{align*}
Since $(u^{\rm inc}+u^{\rm r})|_{t=0}=\partial_{t}(u^{\rm inc}+u^{\rm r})|_{t=0}=0$ in $\Omega$,
 we have $\mathscr L(\partial_{t} ( u^{\rm inc} + u^{\rm r}))=s(\breve{u}^{\rm inc} +u^{\rm r})$ in $\Omega$. It is easy to note that
\begin{align*}
|s|^{2}(\breve{u}^{\rm inc}+ \breve u^{\rm r})=(2s_{1}-s)s(\breve{u}^{\rm inc} + \breve u^{\rm r})
=2s_{1}\mathscr L ( \partial_{t}u^{ \rm inc}+\partial_{t}u^{\rm r})-\mathscr L(\partial_{t}^{2}u^{\rm inc} +\partial_t^2 u^{\rm r} ),
\end{align*}
and
\begin{align*}
|s|^{2}(\nabla\breve{u}^{\rm inc}+\nabla \breve u^{\rm r})=2s_{1}\mathscr L(\partial_{t}\nabla u^{\rm inc} + \partial_t \nabla u^{\rm r})
-\mathscr L(\partial_{t}^{2}\nabla u^{\rm inc} +\partial_t^2 \nabla u^{\rm r}).
\end{align*}
Hence we have
\begin{align*}
\int_0^{\infty}
& e^{-2 s_1 t}
 \left( \|\nabla u\|_{L^2(\Omega)^2}^2
+\|\partial_t u\|^2_{L^2(\Omega)}\right) {\rm d} t\\
 &\lesssim  \int_{-\infty}^{\infty} \|\mathscr L (\partial_{t} u^{\rm inc} +\partial_t u^{\rm r})\|^2_{L^2 (\Omega)}{\rm d} s_2+s_{1}^{-2}\int_{-\infty}^{\infty}\|\mathscr L(\partial^{2}_{t}  u^{\rm inc} +\partial_t^2 u^{\rm r})\|^2_{L^{2} (\Omega)} {\rm d} s_2\\
  &+ \int_{-\infty}^{\infty} \|\mathscr L \left(\partial_{t}\nabla u^{\rm inc} +  \partial_t \nabla u^{\rm r}) \right)\|^2_{L^2 (\Omega)^{2}}{\rm d} s_2+s_{1}^{-2}\int_{-\infty}^{\infty}\|\mathscr L\left(\partial^{2}_{t}  \nabla u^{\rm inc} +\partial_t^2  \nabla u^{\rm r} \right)\|^2_{L^{2} (\Omega)^{2}} {\rm d} s_2.
\end{align*}
Using the Parseval identity \eqref{PI} again gives
\begin{align*}
\int_0^{\infty} &e^{-2 s_1 t}
 \left( \|\nabla u\|_{L^2(\Omega)^2}^2
+\|\partial_t u\|^2_{L^2(\Omega)}\right) {\rm d} t\\
&\lesssim  \int_{0}^{\infty} e^{-2 s_1 t}\| (\partial_{t} u^{inc} +\partial_t u^{\rm r})\|^2_{H^{1} (\Omega)} {\rm d} t+s_{1}^{-2}\int_{0}^{\infty}e^{-2 s_1 t}\|(\partial^{2}_{t}  u^{\rm inc} +\partial_t^2 u^{\rm r})\|^2_{H^{1} (\Omega)} {\rm d} t,
\end{align*}
which shows
\begin{align*}
  u\in L^2 \left( 0, T; H^1_{\rm S} (\Omega) \right) \cap H^1 \left( 0, T; L^2 (\Omega) \right).
\end{align*}

Next, we prove the stability. For any $0<t<T$, define the energy function
\begin{align*}
 e_1 (t) =\| \varepsilon^{1/2} \partial_t u (\cdot, t)\|^2_{L^2 (\Omega)} +\|  \mu^{-1/2} \nabla u (\cdot, t)\|^2_{L^2 (\Omega)^2}.
\end{align*}
It follows from \eqref{au2} and integration by parts that
\begin{align*}
 e_1(t) -e_1 (0)
 =\int_0^t e_1'(\tau ) {\rm d} \tau
 &= 2 {\rm Re}\int_0^t \int_{\Omega} \left(  \nabla \cdot \left(  \mu^{-1} \nabla u\right) \partial_t \bar u +\mu^{-1} (\nabla \partial_t u) \cdot \nabla \bar u \right) {\rm d} \boldsymbol \rho {\rm d} \tau.
\end{align*}
Since $e_{1}(0)=0$, we obtain from Lemma \ref{op2} that
\begin{align}
e_1(t)=\int_0^t e_1'(\tau ) {\rm d} \tau
 =&2 {\rm Re} \int_0^t \int_{\Omega} \left(  - \mu^{-1} \nabla u \cdot (\nabla \partial_t \bar u)
 +\mu^{-1} (\nabla \partial_t u) \cdot \nabla \bar u \right) {\rm d} \boldsymbol \rho {\rm d} \tau \nonumber\\
 &+2{\rm Re} \int_0^t \int_{\Gamma}\mu^{-1}(\partial_{r}u) \partial_{t}\bar{u}{\rm d}\gamma{\rm d}t \nonumber \\
 =& 2{\rm Re} \int_0^t\mu^{-1}\langle\mathscr T u, \partial_{t}u\rangle_{\Gamma}{\rm d}t+2{\rm Re} \int_0^t\langle g,\partial_{t}u\rangle_{\Gamma}{\rm d}t \nonumber\\
 \leq& 2{\rm Re} \int_0^t(\|g\|_{H^{-1/2}(\Gamma)}\|\partial_{t}u\|_{H^{1/2}(\Gamma)}){\rm d}t \nonumber\\
 \lesssim&2{\rm Re}\int_0^t(\|g\|_{H^{-1/2}(\Gamma)}\|\partial_{t}u\|_{H^{1}(\Omega)}){\rm d}t \nonumber\\
 \leq&2\left(\max\limits_{t\in[t,T]}\|\partial_{t}u\|_{H^{1}(\Omega)}\right)\|g\|_{L^{1}(0, T; H^{-1/2}(\Gamma))} . \label{ieq}
\end{align}
Since  the right-hand side of \eqref{ieq} contains the term $\partial_t \nabla u$, which can not be controlled by the left-hand side
of \eqref{ieq}, hence we need to consider a new reduced system.
Taking the derivative of \eqref{au2} with respect to $t$, we know that $\partial_{t}u$ also satisfies the same equations with $g$ replaced by $\partial_{t}g$.
Hence we may consider the similar energy function
\begin{align*}
e_2 (t)= \|\varepsilon^{1/2} \partial_t^{2} u (\cdot, t)\|^2_{L^2 (\Omega)} +\|\mu^{-1/2} \partial_{t}(\nabla u (\cdot, t))\|^2_{L^2 (\Omega)^2},
\end{align*}
and get the estimate
\begin{align*}
e_2 (t)
&\leq2{\rm Re} \int_0^t\int_\Gamma (\partial_{t}g) \partial^{2}_{t}\bar{u} {\rm d}\gamma{\rm d}t\\
&=2{\rm Re} \int_\Gamma(\partial_{t}g)\partial_{t}\bar{u}|^{t}_{0}{\rm d}\gamma-2{\rm Re} \int_0^t\int_\Gamma(\partial^{2}_{t}g) \partial_{t}\bar{u} {\rm d}\gamma{\rm d}t\\
&\leq2{\rm Re}(\max\limits_{t\in[t,T]}\|\partial_{t}u\|_{H^{1}(\Omega)})(\max\limits_{t\in[t,T]}\|\partial_{t}g\|_{H^{-1/2}(\Gamma)}+\|\partial^{2}_{t}g\|_{L^{1}(0, T; H^{-1/2}(\Gamma))}).
\end{align*}
Combing  above estimates, we can obtain
\begin{align*}
\max\limits_{t\in[t,T]}\|\partial_{t}u\|^{2}_{H^{1}(\Omega)}\lesssim(\|g\|_{L^{1}(0, T; H^{-1/2}(\Gamma))}+\max\limits_{t\in[t,T]}\|\partial_{t}g\|_{H^{-1/2}(\Gamma)}+\|\partial^{2}_{t}g\|_{L^{1}(0, T; H^{-1/2}(\Gamma))})\|\partial_{t}u\|_{H^{1}(\Omega)},
\end{align*}
which give the estimate \eqref{UU1} after applying Young's inequality inequality.
\end{proof}
\subsection{A priori estimates of one cavity problem}

In this section, we derive a priori estimates for the total field with a minimum regularity requirement for the data and an explicit dependence on the time.

The variation problem of \eqref{au2} in time-domain is to find $u\in H^{1}_{\rm S}(\Omega)$ for all $t>0$ such that
\begin{align}\label{ttt}
\int_{\Omega}\varepsilon(\partial^{2}_{t}u)\bar{v}{\rm d}\boldsymbol\rho=-\int_{\Omega}\mu^{-1}\nabla u\cdot\nabla\bar{v}{\rm d}\boldsymbol\rho+\int_\Gamma\mu^{-1}(\mathscr{T}u)\bar{v}{\rm d}\gamma+\int_\Gamma g\bar{v}{\rm d}\gamma, \quad \forall  v\in H^1_{S} (\Omega).
\end{align}
To show the stability of its solution, we follow the argument in \cite{04} but with careful study of the TBC.
\begin{theo}\label{ES4}
Let $u\in H^{1}_{\rm S}(\Omega)$ be the solution of \eqref{au2}. Given $g\in L^{1}(0, T; H^{-1/2}(\Gamma))$, we have for any $T>0$ that
\begin{align}\label{we}
  \|u\| _{L^{\infty} (0, T; L^2 (\Omega))}+\|\nabla u\| _{L^{\infty} (0, T; L^2 (\Omega)^2)}
  &\lesssim T \|g\|_{L^1 (0,T; H^{-1/2}(\Gamma))} +\|\partial_{t}g\|_{L^1 (0, T; H^{-1/2}(\Gamma))},
  \end{align}
  and
  \begin{align}\label{MM}
 \|u\|_{L^{2} (0, T; L^2 (\Omega))}+\|\nabla u\| _{L^{2} (0, T; L^2 (\Omega)^2)}
 &\lesssim T^{3/2}\|g\|_{L^1 (0,T; H^{-1/2}(\Gamma))} +T^{1/2} \|\partial_{t}g\|_{L^1 (0, T; H^{-1/2}(\Gamma))}.
\end{align}
\end{theo}
\begin{proof}
Let $ 0< \theta < T$ and define  an  auxiliary function
 \begin{align*}
  \psi_{1} (\boldsymbol \rho, t) =\int_t^{\theta} u (\boldsymbol \rho, \tau) {\rm d} \tau, \quad \boldsymbol \rho  \in \Omega, ~~~ 0 \leq t \leq \theta.
 \end{align*}
It is clear that
\begin{align}\label{F1}
 \psi _{1} (\boldsymbol \rho, \theta)=0, \quad \partial_t \psi _{1} (\boldsymbol \rho, t) =- u(\boldsymbol \rho, t).
\end{align}
For any $\phi (\boldsymbol \rho, t) \in L^2 \left( 0, \theta;~ L^2 (\Omega) \right)$, we have
\begin{align}\label{F2}
 \int_0^{\theta} \phi (\boldsymbol \rho, t) \bar {\psi}_{1} (\boldsymbol \rho, t) {\rm d} t=
 \int_0^{\theta} \left( \int_0^t \phi (\boldsymbol \rho, \tau) {\rm d} \tau \right) \bar u (\boldsymbol \rho, t) {\rm d} t.
\end{align}
Indeed, using integration by parts and \eqref{F1}, we have
\begin{align*}
  \int_0^{\theta}
\phi (\boldsymbol \rho, t) \bar \psi_{1} (\boldsymbol \rho, t) {\rm d} t
 =&\int_0^{\theta}\int_t^{\theta} \bar u (\boldsymbol \rho, \tau) {\rm d} \tau {\rm d} \left( \int_0^{t} \phi (\boldsymbol \rho, \varsigma) {\rm d}\varsigma\right)\\
 =&\int_{t}^{\theta} \bar u (\boldsymbol \rho, \tau) {\rm d} \tau \int_0^{t} \phi (\boldsymbol \rho, \varsigma) {\rm d} \varsigma \big|_0^{\theta}
 +\int_0^{\theta} \left( \int_0^{t} \phi (\boldsymbol \rho,\varsigma) {\rm d} \varsigma \right) \bar u (\boldsymbol \rho, t) {\rm d} t\\
= &\int_0^{\theta} \left( \int_0^t \phi (\boldsymbol \rho, \tau) {\rm d} \tau \right) \bar u (\boldsymbol \rho, t) {\rm d} t.
\end{align*}
Next, we take the test function $v =\psi_{1}$ in \eqref{ttt} and get
\begin{align}\label{tf}
 \int_{\Omega} \varepsilon (\partial_t^2 u) ~ \bar \psi _{1}{\rm d} \boldsymbol  \rho
 = -\int_{\Omega} \mu^{-1}\nabla u \cdot \nabla \bar {\psi} {\rm d} \boldsymbol \rho  +\int_{\Gamma} \mu^{-1}(\mathscr T u) \bar \psi_{1} {\rm d} \gamma
 +\int_{\Gamma}g \bar {\psi}_{1} {\rm d} \gamma.
\end{align}
It follows from the facts in \eqref{F1}  and the initial conditions  in \eqref{au2} that
\begin{align*}
 {\rm Re}\int_0^{\theta} \int_{\Omega}
 \varepsilon (\partial_t^2  u ) \bar \psi_{1}{\rm d} \boldsymbol  \rho  {\rm d} t
 &={\rm Re} \int_{\Omega} \varepsilon \left(  (\partial_t u ) \bar {\psi}_{1} \bigg|_{0}^{\theta} +\frac{1}{2} |u|^2 \bigg|_{0} ^{\theta}\right) {\rm d} \boldsymbol  \rho\\
 &= \frac{1}{ 2} \| \varepsilon^{1/2}  u (\cdot, \theta)\|^2_{L^2 (\Omega)}.
\end{align*}
Integrating \eqref{tf} from $t=0$ to $t=\theta$ and taking the real parts yields
\begin{align}
  &\frac{1}{2}
\|\varepsilon^{1/2} u(\cdot, \theta)\|^2_{L^2 (\Omega)} +{\rm Re} \int_0^{\theta} \int_{\Omega} \mu^{-1} \nabla u \cdot \nabla \bar {\psi}_{1} {\rm d} \boldsymbol  \rho  {\rm d} t \nonumber \\
 =& \frac{1}{2} \|\varepsilon^{1/2} u (\cdot, \theta)\|^2_{L^2 (\Omega)} +\frac{1}{2} \int_{\Omega} \mu^{-1} \left| \int_0^{\theta} \nabla u (\cdot, t) {\rm d} t \right|^2 {\rm d} \boldsymbol \rho \nonumber\\
 = &{\rm Re}\int_0^{\theta} \mu^{-1} \langle \mathscr T u,~ \psi_{1} \rangle_{\Gamma} {\rm d }t
 +{\rm Re}\int_0^{\theta} \int_{\Gamma}g \bar \psi_{1}  {\rm d} \gamma{\rm d} t.\label {en1}
\end{align}
In what follows, we estimate the two terms on the right-hand side of \eqref{en1} separately.

By the property \eqref{F2}, we can obtain
\begin{align*}
 {\rm Re} \int_0^{\theta}  \mu^{-1}\langle \mathscr T u, \psi_{1}\rangle_{\Gamma} {\rm d }t
 ={\rm Re} \int_0^{\theta} \int_0^t\left(  \int_{\Gamma}\mu^{-1}\mathscr T u(\cdot, \tau){\rm d}\gamma \right){\rm d} \tau \bar u (\cdot, t)  {\rm d} t.
\end{align*}
Let $\tilde{u}$ be the extension of $u$ with respect to $t$
 in $\mathbb R$ such that $\tilde{u}=0$
outside the interval $[0,\theta]$. We obtain from the Parseval identity and Lemma \ref{op1} that
\begin{align*}
{\rm Re}\int_{0}^{\theta}e^{-2s_{1}t}
&\int_{0}^{t}\left(  \int_{\Gamma}\mu^{-1}\mathscr T u(\cdot, \tau){\rm d}\gamma \right){\rm d} \tau \bar u (\cdot, t)  {\rm d} t
={\rm Re}\int_{\Gamma}\int_{0}^{\infty}e^{-2s_{1}t}(\int_{0}^{t}\mu^{-1}\mathscr T \tilde{u}(\cdot, \tau){\rm d} \tau )\bar{\tilde{u}}(\cdot, t){\rm d} t{\rm d}\gamma\\
=&{\rm Re}\int_{\Gamma}\int_{0}^{\infty}e^{-2s_{1}t}(\int_{0}^{t}\mathscr{L}^{-1}\circ\mu^{-1}\mathscr{B}\circ\mathscr{L}\tilde{u}(\cdot, \tau){\rm d} \tau)\bar{\tilde{u}}(\cdot, t){\rm d}\gamma{\rm d} t\\
=&{\rm Re}\int_{\Gamma}\int_{0}^{\infty}e^{-2s_{1}t}(\mathscr{L}^{-1}\circ (s\mu)^{-1}\mathscr{B}\circ\mathscr{L}\tilde{u}(\cdot, t))\bar{\tilde{u}}(\cdot, t){\rm d}\gamma{\rm d} t\\
=&\frac{1}{2\pi}\int_{-\infty}^{\infty}{\rm Re}\langle(s\mu)^{-1}\mathscr B \breve{\tilde{u}}, \breve{\tilde{u}}\rangle_{\Gamma}{\rm d}s_{2}\leq0,
\end{align*}
where we have used the fact that
\begin{align*}
\int_{0}^{t}u(\tau){\rm d}\tau=\mathscr{L}^{-1}(s^{-1}\breve{u}(s)).
\end{align*}
After taking $s_{1}\rightarrow 0$, we obtain
\begin{align} \label{p1}
{\rm Re} \int_0^{\theta}  \mu^{-1}\langle \mathscr T u, \psi_{1}\rangle_{\Gamma} {\rm d }t\leq0.
\end{align}
For $0 \leq t \leq \theta \leq T,$  by \eqref{F2} we have
\begin{align}
 {\rm Re}\int_0^{\theta}\int_{\Gamma} g \bar {\psi}_{1} {\rm d} \gamma{\rm d}t
 &={\rm Re}\int_0^{\theta}\left(\int_0^t \int_{\Gamma}g (\cdot, \tau)   {\rm d }\gamma {\rm d} \tau\right)\bar u (\cdot, t) {\rm d} t \nonumber\\
 &\leq \int_0^{\theta} \left( \int_0^t \|g (\cdot, \tau)\|_{H^{-1/2}(\Gamma)} {\rm d} \tau\right) \|u (\cdot, t)\|_{H^{1/2}(\Gamma)} {\rm d} t\nonumber\\
 &\leq \int_0^{\theta} \left( \int_0^t \|g (\cdot, \tau)\|_{H^{-1/2}(\Gamma)} {\rm d} \tau\right) \|u(\cdot, t)\|_{H^{1}(\Omega)} {\rm d} t\nonumber\\
 &\leq  \left( \int_0^\theta \|g (\cdot, t)\|_{H^{-1/2}(\Gamma)} {\rm d} t\right) \left( \int_0^\theta \|u(\cdot, t)\|_{H^{1}(\Omega)} {\rm d} t\right). \label{p2}
\end{align}
Combining \eqref{en1}--\eqref{p2}, we have for any $\theta \in [0, T]$ that
 \begin{align}
 \frac{1}{2} \|\varepsilon^{1/2} u (\cdot, \theta)\|^2_{L^2 (\Omega)}
 &\leq  \frac{1}{2} \|\varepsilon^{1/2} u (\cdot, \theta)\|^2_{L^2 (\Omega)}
 +\frac{1}{2} \int_{\Omega} \mu^{-1} \left| \int_0^{\theta} \nabla u (\cdot, t) {\rm d} t \right|^2 {\rm d} \boldsymbol \rho  \nonumber\\
 &\leq \left( \int_0^\theta \|g(\cdot, t)\|_{L^2 (\Omega)} {\rm d} t\right)
 \left( \int_0^\theta \|u(\cdot, t)\|_{L^2 (\Omega)} {\rm d} t\right). \label{p4}
 \end{align}
Taking the derivative of \eqref{au2} with respect to $t$, we know that $\partial_{t}u$ satisfies the same equation with $g$ replaced by $\partial_t g$. Define
\begin{align*}
  \psi_{2} (\boldsymbol \rho, t) =\int_t^{\theta} \partial_t u(\boldsymbol \rho, \tau) {\rm d} \tau, \quad \boldsymbol \rho  \in \Omega, ~~~ 0 \leq t \leq \theta.
 \end{align*}
 We may follow the same steps as those for $\psi_{1}$ to obtain
\begin{align}
 &\frac{1}{2}
 \|\varepsilon^{1/2}\partial_{t} u(\cdot, \theta)\|^2_{L^2 (\Omega)} +\frac{1}{2} \int_{\Omega} \mu^{-1} \left| \int_0^{\theta} \partial_{t}(\nabla u (\cdot, t)) {\rm d} t \right|^2 {\rm d} \boldsymbol \rho  \nonumber\\
=  &{\rm Re}\int_0^{\theta} \mu^{-1} \langle \mathscr T\partial_{t}u,~ \psi_{2} \rangle_{\Gamma} {\rm d }t
 +{\rm Re}\int_0^{\theta} \int_{\Gamma} (\partial_t g)\bar \psi_{2}  {\rm d} \gamma{\rm d} t.
 \label {en2}
\end{align}
Integrating by parts yields that
 \begin{align}
\frac{1}{2} \int_{\Omega} \mu^{-1} \left| \int_0^{\theta} \partial_{t}(\nabla u (\cdot, t)) {\rm d} t \right|^2 {\rm d} \boldsymbol \rho=\frac{1}{2} \|\mu^{-1/2} \nabla u (\cdot, \theta) \|^2 _{L^{2}(\Omega)}.\label{p7}
\end{align}
The first term on the right-hand side of \eqref{en2} can be discussed as above, we only consider the second term. By \eqref{F1} and  Lemma \ref{tt10}, we get
\begin{align}
\int_0^{\theta}\int_{\Gamma} (\partial_{t}g) \bar {\psi}_{2} {\rm d} \gamma{\rm d}t
 &= \int_{\Gamma}( \int_0^t \partial_{\tau}g (\cdot, \tau) {\rm d} \tau)\bar{u} (\cdot, t)|^{\theta}_{0} {\rm d} \gamma-\int_0^{\theta}\int_{\Gamma}(\partial_{t}g (\cdot, t))u(\cdot, t)  {\rm d} \gamma{\rm d} t \nonumber\\
 &\lesssim \int_0^{\theta}\|\partial_{t}g (\cdot, t)\|_{H^{-1/2}(\Gamma)}\|u(\cdot, t)\|_{H^{1/2}(\Gamma)} {\rm d} t\nonumber\\
 &\lesssim  \int_0^{\theta}\|\partial_{t}g (\cdot, t)\|_{H^{-1/2}(\Gamma)}\|u(\cdot, t)\|_{H^{1}(\Omega)} {\rm d} t.\label{p6}
\end{align}
Substituting \eqref{p7}--\eqref{p6} into \eqref{en2}, we have for any $\theta \in [0, T]$ that
 \begin{align}
 \frac{1}{2} \|\varepsilon^{1/2} \partial_{t}u (\cdot, \theta)\|^2_{L^2 (\Omega)}+\frac{1}{2} \|\mu^{-1/2} \nabla u (\cdot, \theta)) \|^2 _{L^{2}(\Omega)^2}
 &\lesssim  \int_0^{\theta}\|\partial_{t}g (\cdot, t)\|_{H^{-1/2}(\Gamma)}\|u(\cdot, t)\|_{H^{1}(\Omega)} {\rm d} t. \label{p5}
 \end{align}
Combing the estimates \eqref{p4} and \eqref{p5}, it follows
\begin{align}
\|u(\cdot, \theta)\|^{2}_{L^{2}(\Omega)}+\|\nabla u\|^{2}_{L^{2}(\Omega)^{2}}
 &\lesssim (\int^{\theta}_{0}\|g(\cdot, t)\|_{H^{-1/2}(\Gamma)}{\rm d}t)(\int^{\theta}_{0}\|u(\cdot, t)\|_{H^{1}(\Omega)}{\rm d}t)\nonumber\\
 &+\int^{\theta}_{0}\|\partial_{t}g(\cdot, t)\|_{H^{-1/2}(\Gamma)}\|u(\cdot, t)\|_{H^{1}(\Omega)}{\rm d}t).\label{emm}
\end{align}
Taking the $L^{\infty}$-norm with respect to $\theta$ on both sides of \eqref{emm} yields
\begin{align*}
 \|u\|^2_{L^{\infty} (0, T; L^2 (\Omega))}+ \|\nabla u\|^{2}_{L^{\infty}(0, T; L^2 (\Omega)^{2})}
 &\lesssim T \|g\|_{L^1 (0, T; H^{-1/2 } (\Gamma))}\|u\|_{L^{\infty} (0, T; H^1 (\Omega))} \nonumber\\
 &+\|\partial_{t}g\|_{L^1 (0, T; H^{-1/2 } (\Gamma))}\|u\|_{L^{\infty} (0, T; H^1 (\Omega))},
\end{align*}
which gives the estimate \eqref{we} after applying the Young's inequality.

Integrating \eqref{emm} with respect to $\theta$ from $0$ to $T$ and using the Cauchy--Schwarz inequality, we obtain
\begin{align*}
 \|u\|^2_{L^{2} (0, T; L^2 (\Omega))}+ \|\nabla u\|^{2}_{L^{2}(0, T; L^2 (\Omega)^{2})}
 &\lesssim T^{3/2} \|g\|_{L^1 (0, T; H^{-1/2 } (\Gamma))}\|u\|_{L^{2} (0, T; H^1 (\Omega))} \nonumber\\
 &+T^{1/2}\|\partial_{t}g\|_{L^1 (0, T; H^{-1/2 } (\Gamma))}\|u\|_{L^{2} (0, T; H^1 (\Omega))},
\end{align*}
which implies the estimate \eqref{MM} by using  Young's inequality again.
\end{proof}
\section{two cavities scattering problem}\label{TSP}
In order to address the general multiple cavity scattering problem,
in this section, we first give the discussion on the two cavity scattering problem.
 As it shows that the two cavity scattering problem shares the same features with the general multiple cavity scattering problem,
 but is  easier to present the major ideas in the proof of the
well-posedness and stability for the multiple cavity scattering  problem.

\subsection {Problem formulation}
\begin{figure}
\centering
\includegraphics[width=0.45\textwidth]{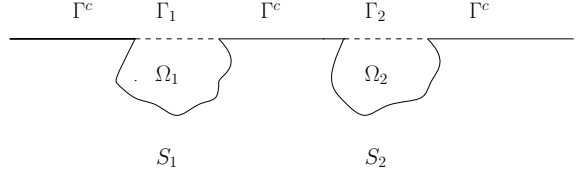}
\caption{The problem geometry of the two cavity.}
\label{fig2}
\end{figure}

As shown in the Figure \ref{fig2}, two open cavities $\Omega_{1}$ and $\Omega_{2}$,
enclosed by the apertures $\Gamma_{1}$ and $\Gamma_{2}$ and the walls $S_{1}$ and $S_{2}$, are placed on a perfectly conducting ground plane $\Gamma^{\rm c}$. Above the flat surface $\Gamma^{\rm c}\cup\Gamma_{1}\cup\Gamma_{2}$, the medium is assumed to be homogeneous with positive dielectric permittivity $\varepsilon_{0}$ and magnetic permeability $\mu_{0}$. The medium inside the cavity $\Omega_{1}$ and $\Omega_{2}$ is inhomogeneous with a variable dielectric permittivity $\varepsilon_{j}(x,y)$ , respectively and the same  variable magnetic permeability $\mu(x,y).$
Assume further that $\varepsilon_{j}(x,y)\in L^{\infty} (\Omega_j)$ and $\mu(x,y)\in L^{\infty} (\Omega_j)$  and  satisfy
 \[
 0 < \varepsilon_{j, \rm min} \leq \varepsilon_j \leq \varepsilon_{j, \rm max}<
\infty, \quad 0 < \mu_{\rm min} \leq \mu\leq \mu_{\rm max} < \infty ~\quad \text {for}~j=1,2.
 \]

\subsubsection{Transparent boundary condition}
In TE polarization, the three-dimensional Maxwell equations can be  reduced to the two-dimensional wave equation with initial-boundary value problem
\begin{align}\label{a3.1}
 \begin{cases}
  \varepsilon \partial_t^2 u - \nabla \cdot \left(  \mu^{-1} \nabla u\right)=0 \quad  &\text {in} ~~~\Omega^{\rm e}\cup\Omega_1\cup \Omega_2,~~~t>0,\\
  u \big|_{t=0}=0, \quad \partial_t u \big|_{t=0} =0\quad & \text {in}~~~\Omega^{\rm e}\cup\Omega_1\cup \Omega_2,\\
  u=0\quad &\text{on} ~~ \Gamma^{\rm c}\cup S_1 \cup S_2  , ~~t>0.
\end{cases}
\end{align}
Let the plane wave $u^{\rm inc}$ be incident on the cavities from above.
 Due to the interaction between the incident wave and the ground plane and the  two cavity, it can be shown that
 the total field $u$ is composed  of the incident field $u^{\rm inc}$, the reflected field $u^{\rm r}$ and the scattered field $u^{\rm sc}$. The scattered  field $u^{\rm sc}$ is also  required to satisfy the radiation condition \eqref{SR}.

To reduce the scattering problem from the open domain $\Omega^{\rm e}\cup\Omega_1\cup \Omega_2$ into the bounded domain, we need to derive transparent boundary conditions on the apertures $\Gamma_{1}$ and $\Gamma_{2}$.
We want to reduce \eqref{a3.1} into two single cavity scattering problem: for $ j=1,2$,
\begin{align}\label{aab}
 \begin{cases}
  \varepsilon_{j} \partial_t^2 u_{j} - \nabla \cdot \left(  \mu^{-1} \nabla u_{j}\right)=0 \quad  &\text {in} ~~~\Omega_{j},~~t>0,\\
  u_{j} \big|_{t=0}=0, \quad \partial_t u_{j} \big|_{t=0} =0\quad & \text {in}~~\Omega_{j},\\
  u_{j}=0  \quad & \text {on}~~S_{j}, ~~t>0,\\
  \partial_{\boldsymbol n} u_{j} = \mathscr T u_{j}+g \quad & \text {on}~~\Gamma_{j},~~t>0,
\end{cases}
\end{align}
where $\mathscr T$ is transparent boundary conditions in time-domain.
 Obviously, if $u$ is the solution of \eqref{a3.1},  then  $u_j$  are solutions of \eqref{aab}. Moreover, it has
 $u|_{\Omega_{j}}=u_{j}$.

Due to the homogeneous medium in the upper half space $\Omega^{\rm e}$ and the radiation condition \eqref{SR}, after taking the Laplace transform with respect to $t$,
 the scattered field $\breve u^{\rm sc}$ still satisfies the same ordinary differential equation \eqref{abc}.
   Thus, in $s$-domain and in the time-domain,  the transparent boundary condition can be respectively written as
\begin{align}\label{abd}
\partial_{\mathbf{n}}\breve{u}=\mathscr B \breve{u} +\breve{g}, \quad\partial_{\mathbf{n}}u=\mathscr T u +g  ~\quad &\text {on}~~~\Gamma^{\rm c}\cup\Gamma_{1}\cup\Gamma_{2}.
\end{align}

For and  $u_{j}(x,0) (j=1,2)$ defined on $\Gamma_j$,
define the extension  to the whole $x$-axis by
\begin{align*}
\tilde{u}_{j}(x,0)=
 \begin{cases}
  u_{j}(x,0) ~~\quad & \text{for} ~x\in \Gamma_{j},\\
  0          ~~\quad & \text{for} ~x\in \mathbb R\backslash\Gamma_{j}.
\end{cases}
\end{align*}
For the total field $u(x,0)$, define its extension to the whole $x$-axis by
\begin{align*}
\tilde{u}(x,0)=
 \begin{cases}
  u_{1}(x,0) ~~\quad & \text{for} ~x\in \Gamma_{1},\\
  u_{2}(x,0) ~~\quad & \text{for} ~x\in \Gamma_{2},\\
  0          ~~\quad & \text{for} ~x\in \Gamma^{\rm c}.
\end{cases}
\end{align*}
It follows from the definitions of these extensions that
\begin{align*}
\tilde{u}=\tilde{u}_{1}+\tilde{u}_{2} ~~\quad \text{on}~ \Gamma^{\rm c}\cup\Gamma_{1}\cup\Gamma_{2}.
\end{align*}
The transparent boundary conditions  \eqref{abd} can be respectively  written  as
\begin{align*}
\partial_{\mathbf{n}}\breve{ \tilde u}=\mathscr B {\breve{ \tilde u}} +\breve{g}  ~\quad \text {on}~~~\Gamma^{\rm c}\cup\Gamma_{1}\cup\Gamma_{2}; \quad
\partial_{\mathbf{n}}{\tilde u}=\mathscr T \tilde{u} +g  ~\quad \text {on}~~~\Gamma^{\rm c}\cup\Gamma_{1}\cup\Gamma_{2}, ~t>0.
\end{align*}
These  lead to the transparent boundary conditions for $u_{j}$ on $\Gamma_j$ in  frequency domain and time-domain, respectively:
\begin{align}\label{aae}
\partial_{\mathbf{n}}\breve{ u}_{1}=\mathscr B{\breve{ \tilde u}}_{1}+\mathscr B{\breve{ \tilde u}}_{2}+\breve{g} ~~\quad \text{on} ~\Gamma_{1};\quad
\partial_{\mathbf{n}}u_{1}=\mathscr T \tilde{u}_{1}+\mathscr T \tilde{u}_{2}+g ~~\quad \text{on} ~\Gamma_{1},~t>0,
\end{align}
and
\begin{align}\label{aaf}
\partial_{\mathbf{n}}\breve{u}_{2}=\mathscr B{\breve{\tilde u}}_{2}+\mathscr B\tilde{\breve{u}}_{1}+\breve{g} ~~\quad \text{on}~ \Gamma_{2};
\quad
\partial_{\mathbf{n}}u_{2}=\mathscr T \tilde{u}_{2}+\mathscr T \tilde{u}_{1}+g ~~\quad \text{on}~ \Gamma_{2},~t>0.
\end{align}
From \eqref{aae} and \eqref{aaf}, we find the boundary conditions for $u_{1}$ and $u_{2}$ are coupled with each other, which is the major difference between the single cavity scattering problem.

The following two lemmas are  analogous to Lemmas \ref{op1}--\ref{op2}, which
 will be used to analysis the uniqueness and existence for the solution of the two cavity scattering problem.
\begin{lemm}\label{aag}
 It holds that
 \begin{align*}
  -{\rm Re}\left( \langle (s\mu)^{-1} \mathscr B u, u \rangle_{\Gamma_1}+\langle(s\mu)^{-1} \mathscr B v, u \rangle_{\Gamma_1}+\langle(s\mu)^{-1} \mathscr B v, v \rangle_{\Gamma_2}+\langle(s\mu)^{-1} \mathscr B u, v \rangle_{\Gamma_2}\right) \geq 0, \quad u,v\in H^{1/2} (\mathbb R).
 \end{align*}

\end{lemm}
\begin{proof}
Recalling $\beta^{2}(\xi)=\xi^{2}+c^{-2}s^{2}$ and using \eqref{ht},
we get
\begin{align*}
{\rm Re}&\left(\langle (s\mu)^{-1} \mathscr B u, u \rangle_{\Gamma_1}+\langle(s\mu)^{-1} \mathscr B v, u \rangle_{\Gamma_1}+\langle(s\mu)^{-1} \mathscr B v, v \rangle_{\Gamma_2}+\langle(s\mu)^{-1} \mathscr B u, v \rangle_{\Gamma_2}\right) \\
=& {\rm Re}\int_{\mathbb R}(s\mu)^{-1}\beta(\xi)(|u|^{2}+|v|^{2}+ v\bar{u}+u\bar{v}){\rm d}\xi\\
=&\int_{\mathbb R}\frac{1}{\mu|s|^{2}}\frac{s_{1}}{\zeta}(\zeta^{2}+c^{-2}s_{2}^{2})(|u+v|^{2}){\rm d}\xi\leq 0,
\end{align*}
 where $\beta(\xi)=\zeta+{\rm i}\varrho$ with $\zeta < 0$.

\end{proof}

\begin{lemm}\label{aah}
 For any $u(\cdot, t) \in L^2 (0, T; H^{1/2} (\Gamma_1)), v(\cdot, t)\in L^2 (0, T; H^{1/2} (\Gamma_2))$ with initial values $u (\cdot, 0)=0,v (\cdot, 0)=0$, denote their zero extension on $ L^2 (0, T; H^{1/2} (\mathbb R))$ by $\tilde u (\cdot, t)$ and $\tilde v (\cdot, t)$, respectively. Then,
  it holds that
 \begin{align*}
  - {\rm Re}\int_0^T \left(\langle \mathscr T \tilde u , \partial_t \tilde u  \rangle_{\Gamma_1} +\langle \mathscr T \tilde v , \partial_t \tilde  u  \rangle_{\Gamma_1}+\langle \mathscr T \tilde v , \partial_t \tilde v \rangle_{\Gamma_2}+\langle \mathscr T \tilde u , \partial_t \tilde v  \rangle_{\Gamma_2}\right) {\rm d} t \geq 0.
 \end{align*}

\end{lemm}
\begin{proof}
 Let $\tilde {\tilde  u} (\cdot, t),\tilde{ \tilde v}(\cdot, t)$ be the extension of $\tilde u (\cdot, t),  \tilde v (\cdot, t)$ with respect to $t$ in $\mathbb R$ such that
 $\tilde {\tilde u} (\cdot, t)= 0, \tilde {\tilde  v} (\cdot, t)= 0$ outside the interval $[0, T]$, and $\breve {\tilde {\tilde u}} =\mathscr L (\tilde {\tilde u}), \breve {\tilde {\tilde v}} =\mathscr L (\tilde {\tilde v})$ be the Laplace transform  of $\tilde {\tilde u}, \tilde {\tilde v}$. By the Parseval identity \eqref{PI}, we get
 \begin{align*}
  &-{\rm Re} \int_0^{T} e^{- 2 s_1 t} \left(\langle \mathscr T \tilde u, \partial_t \tilde u \rangle_{\Gamma_1}+\langle \mathscr T \tilde v, \partial_t \tilde u \rangle_{\Gamma_1}+\langle \mathscr T \tilde v, \partial_t \tilde v \rangle_{\Gamma_2}+\langle \mathscr T \tilde u, \partial_t \tilde v \rangle_{\Gamma_2}\right) {\rm d} t \\
  =& - {\rm Re}\overset{2}{\underset{j=1}{\sum}} \left(\int_{\Gamma_{j}} \int_0^{\infty} e^{- 2 s_1 t} (\mathscr T \tilde {\tilde u}) \partial_t \bar {\tilde {\tilde u}} {\rm d} t{\rm d} \gamma_{j}+\int_{\Gamma_{j}} \int_0^{\infty} e^{- 2 s_1 t} (\mathscr T \tilde {\tilde v}) \partial_t \bar {\tilde {\tilde u}} {\rm d} t{\rm d} \gamma_{j}\right)\\
  =&-\frac{1}{2 \pi} \int_{-\infty}^{\infty}  {\rm Re}\left( \langle \mathscr B \breve {\tilde {\tilde u}}, s \breve{\tilde {\tilde u}} \rangle_{\Gamma_{1}}+\langle \mathscr B \breve {\tilde {\tilde v}}, s \breve{\tilde {\tilde u}} \rangle_{\Gamma_{1}}+\langle \mathscr B \breve {\tilde {\tilde v}}, s \breve{\tilde {\tilde  v}} \rangle_{\Gamma_{2}}+\langle \mathscr B \breve {\tilde {\tilde  u}}, s \breve{\tilde {\tilde v}} \rangle_{\Gamma_{2}}\right){\rm d} s_2\\
  =&-\frac{1}{2 \pi} \int_{-\infty}^{\infty} |s|^{2} {\rm Re}\left( \langle s^{-1}\mathscr B \breve {\tilde {\tilde u}},  \breve{\tilde {\tilde u}} \rangle_{\Gamma_{1}}+\langle s^{-1}\mathscr B \breve {\tilde {\tilde v}}, \breve{\tilde {\tilde u}} \rangle_{\Gamma_{1}}+\langle s^{-1}\mathscr B \breve {\tilde {\tilde v}},  \breve{\tilde {\tilde v}} \rangle_{\Gamma_{2}}+\langle s^{-1} \mathscr B \breve {\tilde {\tilde u}},  \breve{\tilde {\tilde v}} \rangle_{\Gamma_{2}}\right){\rm d} s_2.
 \end{align*}
 It follows from  Lemma \ref{aag} and $\beta(\xi)=\varsigma+{\rm i}\varrho $ with $\varsigma<0$ that
 \begin{align*}
 -{\rm Re} \int_0^{T} e^{- 2 s_1 t} \left(\langle \mathscr T \tilde u, \partial_t \tilde  u \rangle_{\Gamma_1}+\langle \mathscr T \tilde v, \partial_t \tilde  u \rangle_{\Gamma_1}+\langle \mathscr T \tilde  v, \partial_t  \tilde v \rangle_{\Gamma_2}+\langle \mathscr T \tilde u, \partial_t \tilde v \rangle_{\Gamma_2}\right) {\rm d} t\geq0
 \end{align*}
which  completes the proof
after taking $s_1 \rightarrow 0 $.
\end{proof}

\subsection{The reduced two cavity scattering problem}

In this section, we will discuss the well-posedness and stability  for the reduced problem  of  the two cavity scattering problem.
 Firstly, we denote $\Omega=\Omega_{1}\cup\Omega_{2}, \Gamma=\Gamma_{1}\cup\Gamma_{2}$, and $S=S_{1}\cup S_{2}$. Let
\begin{align*}
u=
\begin{cases}
u_{1} ~~\quad & \text{in} ~\Omega_{1},\\
u_{2} ~~\quad & \text{in} ~\Omega_{2}.
\end{cases}
\end{align*}
Define a trace functional space
\begin{align*}
\tilde{H}^{1/2}(\Gamma)=\tilde{H}^{1/2}(\Gamma_{1})\times\tilde{H}^{1/2}(\Gamma_{2}),
\end{align*}
whose norm is characterized by
$
\|u\|^{2}_{\tilde{H}^{1/2}(\Gamma)}=\|u_{1}\|^{2}_{\tilde{H}^{1/2}(\Gamma_{1})}+\|u_{2}\|^{2}_{\tilde{H}^{1/2}(\Gamma_{2})}.
$
Denote  by
\[H^{-1/2}(\Gamma)=H^{-1/2}(\Gamma_{1})\times H^{-1/2}(\Gamma_{2}),\]
 which is the dual space of $\tilde{H}^{1/2}(\Gamma)$. The norm on the space $H^{-1/2}(\Gamma)$ is characterized by
\begin{align*}
\|u\|^{2}_{H^{-1/2}(\Gamma)}=\|u_{1}\|^{2}_{H^{-1/2}(\Gamma_{1})}+\|u_{2}\|^{2}_{H^{-1/2}(\Gamma_{2})}.
\end{align*}
Define the  space
\begin{align*}
H^{1}_{\rm S}(\Omega)=H^{1}_{\rm S_1}(\Omega_{1})\times H^{1}_{\rm S_2}(\Omega_{2}),
\end{align*}
which is a Hilbert space with norm characterized by
$
\|u\|^{2}_{H_{\rm S}^{1}(\Omega)}=\|u_{1}\|^{2}_{H^{1}_{\rm S_1}(\Omega_{1})}+\|u_{2}\|^{2}_{H^{1}_{\rm S_2}(\Omega_{2})}.
$
\subsubsection{well-posedness in the $s$-domain}
Now we present a variational formulation for the two cavity scattering problem.
For $j=1, 2$, taking the Laplace transform of \eqref{aab},  we get
\begin{align}\label{ac}
\begin{cases}
\varepsilon_j s \breve u_j-\nabla\cdot \left(s^{-1} \mu^{-1} \nabla \breve u_j \right)=0 \quad &\text{in}~~\Omega_j,\\
\breve u_j=0 \quad &\text{on}~~S_j,\\
\partial_{\boldsymbol n} \breve u_j=\mathscr B \breve u_j+\breve g \quad &\text{on}~~\Gamma_j.
\end{cases}
\end{align}
Multiplying the complex conjugate of test function $v_{j}\in H^{1}_{\rm S_{j}}(\Omega_{j}),j=1,2$ on both sides of the first equation of \eqref{ac}, integrating over $\Omega_{j}$, we have
\begin{align*}
\int_{\Omega_{j}}(s\mu)^{-1}\nabla\breve{u}_{j}\nabla\bar{v}_{j}+s\varepsilon_{j}\breve{u}_{j}\bar{v}_{j}{\rm d}\boldsymbol \rho-\overset{2}{\underset{i=1}{\sum}}\langle(s\mu)^{-1}\mathscr B \tilde{\breve{u}}_{i}, \tilde{v}_{j}\rangle_{\Gamma_{j}}=\langle\breve{g}, v_{j}\rangle_{\Gamma_{j}}.
\end{align*}
We deduce the variational formulation for the two cavity scattering problem:
find $\breve u\in H^{1}_{\rm S}(\Omega)$ with  $\breve u|_{\Omega_{j}}=\breve u_{1}\in H^{1}_{\rm S_j}(\Omega_{j})$,
 such that for all $v\in H^{1}_{\rm S}(\Omega)$ with  $v_{j}=v|_{\Omega_{j}}\in H^{1}_{\rm S_j}(\Omega_{j})$, it holds
\begin{align}\label{aak}
a_2(\breve{u}, v)=\langle\breve{g}, v_{1}\rangle_{\Gamma_{1}}+\langle\breve{g}, v_{2}\rangle_{\Gamma_{2}},
\end{align}
where the sesquilinear form
\begin{align*}
a_2(\breve{u}, v)&=\overset{2}{\underset{j=1}{\sum}}\int_{\Omega_{j}}(s\mu)^{-1}\nabla\breve{u}_{j}\nabla\bar{v}_{j}+s\varepsilon_{j}\breve{u}_{j}\bar{v}_{j}{\rm d}\boldsymbol \rho-\overset{2}{\underset{j=1}{\sum}}\overset{2}{\underset{i=1}{\sum}}\langle(s\mu)^{-1}\mathscr B \tilde{\breve{u}}_{i}, \tilde{v}_{j}\rangle_{\Gamma_{j}}.
\end{align*}

\begin{theo}\label{aam}
 The variational problem \eqref{aak} has a unique solution $\breve{u} \in H^1_{\rm S} (\Omega)$ which satisfies
 \begin{align}\label{aaq}
  \|\nabla \breve{u}\|_{L^2 (\Omega)^2} +\|s \breve{u}\|_{L^2 (\Omega)} \lesssim s_1^{-1} \| s\breve{g}\|_{H^{-1/2}(\Gamma)}.
 \end{align}
\end{theo}
\begin{proof}
The continuity  of the sesquilinear follows   directly from the Cauchy--Schwarz inequality, Lemma \ref{ct1} and Lemma \ref{tt10},
 \begin{align*}
  |a_{2} (\breve{u}, v)|
  \leq &\overset{2}{\underset{j=1}{\sum}}\left(\frac{1}{|s| \mu_{\min}} \|\nabla \breve{u}_{j}\|_{L^2(\Omega_{j})^2} \|\nabla v_{j}\|_{L^2 (\Omega_{j})^2}
  +|s|\varepsilon_{\max}\|\breve{u}_{j}\|_{L^2 (\Omega_{j})}\|v_{j}\|_{L^2 (\Omega_{j})}\right)\\
  &+\frac{1}{|s| \mu_{\min}}\overset{2}{\underset{j=1}{\sum}}\overset{2}{\underset{i=1}{\sum}}\|\mathscr B \breve{u}_{i}\|_{H^{-1/2} (\Gamma_{j})} \|v_{j}\|_{H^{1/2}(\Gamma_{j})}\\
   \lesssim & \|\nabla \breve{u}\|_{L^2(\Omega)^2} \|\nabla v\|_{L^2 (\Omega)^2}
  +\|\breve{u}\|_{L^2 (\Omega)}\|v\|_{L^2 (\Omega)}+  \|\breve{u}\|_{H^{1/2} (\Gamma)} \|v\|_{H^{1/2} (\Gamma)}\\
   \lesssim &\|\breve{u}\|_{H^1(\Omega)} \|v\|_{H^1(\Omega)},
 \end{align*}
 where $\varepsilon_{\max}=\max\{\varepsilon_{1}, \varepsilon_{2}\}$.
  It suffices to show the coercivity of  $a_{2} (\breve{u}, v)$.
A simple calculation yields
\begin{align}\label{aan}
 a_{2} (\breve{u}, \breve{u}) &=\overset{2}{\underset{j=1}{\sum}}\int_{\Omega_{j}}(s\mu)^{-1}|\nabla\breve{u}_{j}|^{2}+s\varepsilon_{j}|\breve{u}_{j}|^{2}{\rm d}\boldsymbol \rho-\overset{2}{\underset{j=1}{\sum}}\overset{2}{\underset{i=1}{\sum}}\langle(s\mu)^{-1}\mathscr B \tilde{\breve{u}}_{i}, \tilde{\breve{u}}_{j}\rangle_{\Gamma_{j}}.
\end{align}
Taking the real part of \eqref{aan} and using Lemma \ref{aag}, we get
\begin{align}\label{aao}
 {\rm Re} \left(  a_{2} (\breve{u}, \breve{u}) \right) \geq C_{1} \frac{s_1}{|s|^2} \left( \|\nabla \breve{u}\|^2_{L^2 (\Omega)^2}  +\|s \breve{u}\|^2_{L^2 (\Omega)}\right),
\end{align}
where $C_{1}=\min\{\mu^{-1}_{\max},1\}$.

It follows from the Lax--Milgram lemma that the variational problem \eqref{aak} has a unique solution $\breve{u} \in H^1_{\rm S} (\Omega)$ and satisfies $u|_{\Omega_j}=u_i.$
Moreover, we have from \eqref{aak} that
\begin{align}\label{aap}
 |a_{2}(\breve{u}, \breve{u})| \leq |s|^{-1} \|\breve{g}\|_{H^{-1/2}(\Gamma)} \|s \breve{u}\|_{L^2 (\Omega)}.
\end{align}
Combining \eqref{aao}--\eqref{aap} leads to
\begin{align*}
 \|\nabla \breve{u}\|^2_{L^2 (\Omega)^2}  +\|s \breve{u}\|^2_{L^2 (\Omega)} \lesssim s_1^{-1}  \| s \breve{g}\|_{H^{-1/2}(\Gamma)} \|s \breve{u}\|_{L^2 (\Omega)},
\end{align*}
which completes the proof of estimates of \eqref{aaq} after applying the Cauchy--Schwarz inequality.

\end{proof}

\subsubsection{well-posedness in the time-domain}
Using the time-domain transparent boundary conditions \eqref{aae}--\eqref{aaf}, problem \eqref{a3.1} can be equivalently
reduced to the  initial-boundary value problem
\begin{align}\label{aar}
\begin{cases}
 \varepsilon \partial_t^2 u - \nabla \cdot \left(  \mu^{-1} \nabla u\right)=0 \quad  &\text {in} ~~~\Omega,~~~t>0,\\
u\big|_{t=0}=0, \quad \partial_t u \big|_{t=0} =0\quad & \text {in}~~~\Omega,\\
  u=0  \quad & \text {on}~~~S, ~~~~t>0,\\
 \partial_r u = \mathscr T u+g \quad & \text {on}~~~\Gamma,~~~~t>0.
\end{cases}
\end{align}
\begin{theo}\label{aay}
The initial-boundary problem \eqref{aar} has a unique solution $u$, which satisfies
\begin{align*}
u\in L^{2}(0, T; H^{1}_{S}(\Omega))\cap H^{1}(0, T; L^{2}(\Omega)),
\end{align*}
and the stability estimate
\begin{align} \label{aaz}
 \begin{split}
 \max\limits_{t\in[t,T]}&\left(\|\partial_t u\|_{L^2 (\Omega)}  +\|\partial_{t}(\nabla u)\|_{L^2 (\Omega)^2}\right) \\
 & \lesssim \left(\|g\|_{L^{1}(0, T; H^{-1/2}(\Gamma))}+ \max\limits_{t\in[t,T]}\|\partial_{t}g\|_{H^{-1/2}(\Gamma)}+\|\partial^{2}_{t}g\|_{L^{1}(0, T; H^{-1/2}(\Gamma))}\right).
 \end{split}
\end{align}
\end{theo}
\begin{proof}
Using the similar  way as  one cavity scattering problem, we can get
\begin{align*}
  u\in L^2 \left( 0, T; H^1_{\rm S} (\Omega) \right) \cap H^1 \left( 0, T; L^2 (\Omega) \right).
\end{align*}

Next, we prove the stability.
For any $0<t<T$, define the energy function
\begin{align}\label{E1}
 e_3 (t) =\| \varepsilon^{1/2} \partial_t u (\cdot, t)\|^2_{L^2 (\Omega)} +\|  \mu^{-1/2} \nabla u (\cdot, t)\|^2_{L^2 (\Omega)^2}.
\end{align}
Integrating by parts,
it follows from \eqref{aab} that
\begin{align*}
 \int_0^t e_3'(\tau ) {\rm d} \tau
 =&e_3(t) -e_3 (0)  \\
 = &2 {\rm Re}\int_0^t \int_{\Omega_{1}}  \left( \varepsilon (\partial_t^2 u_{1}) \partial_t \bar u_{1} +\mu^{-1} (\nabla \partial_t u_{1}) \cdot \nabla \bar u_{1} \right){\rm d} \boldsymbol \rho {\rm d} \tau  \\
 &+2 {\rm Re}\int_0^t \int_{\Omega_{2}}\left( \varepsilon (\partial_t^2 u_{2}) \partial_t \bar u_{2} +\mu^{-1} (\nabla \partial_t u_{2}) \cdot \nabla \bar u_{2} \right ){\rm d} \boldsymbol \rho {\rm d} \tau \\
 = &\overset{2}{\underset{j=1}{\sum}}2 {\rm Re}\int_0^t \int_{\Omega_{j}} \left(  \nabla \cdot \left(  \mu^{-1} \nabla u_{j}\right) \partial_t \bar u_{j} +\mu^{-1} (\nabla \partial_t u_{j}) \cdot \nabla \bar u_{j} \right) {\rm d} \boldsymbol \rho {\rm d} \tau.
\end{align*}
Since $e_{3}(0)=0$, we obtain from Lemma \ref{aah} that
\begin{align}\label{E3}
e_3(t)=&\int_0^t e_3'(\tau ) {\rm d} \tau \nonumber \\
 =&2 {\rm Re} \int_0^t \int_{\Omega_{1}} \left(  - \mu^{-1} \nabla u_{1} \cdot (\nabla \partial_t \bar u_{1})
 +\mu^{-1} (\nabla \partial_t u_{1}) \cdot \nabla \bar u_{1} \right) {\rm d} \boldsymbol \rho {\rm d} \tau +2{\rm Re} \int_0^t \int_{\Gamma_{1}}\mu^{-1}(\partial_{r}u_{1}) \partial_{t}\bar{u}_{1}{\rm d}\gamma_{1}{\rm d}t \nonumber \\
 &+2 {\rm Re} \int_0^t \int_{\Omega_{2}} \left(  - \mu^{-1} \nabla u_{2} \cdot (\nabla \partial_t \bar u_{2})
 +\mu^{-1} (\nabla \partial_t u_{2}) \cdot \nabla \bar u_{2} \right) {\rm d} \boldsymbol \rho {\rm d} \tau +2{\rm Re} \int_0^t \int_{\Gamma_{2}}\mu^{-1}(\partial_{r}u_{2}) \partial_{t}\bar{u}_{2}{\rm d}\gamma_{2}{\rm d}t \nonumber \\
 =& 2{\rm Re} \int_0^t\mu^{-1}(\langle\mathscr T u_{1}, \partial_{t}u_{1}\rangle_{\Gamma_{1}}+\langle\mathscr T u_{2}, \partial_{t}u_{1}\rangle_{\Gamma_{1}}){\rm d}t+2{\rm Re} \int_0^t\langle g,\partial_{t}u_{1}\rangle_{\Gamma_{1}}{\rm d}t \nonumber \\
 &+2{\rm Re} \int_0^t\mu^{-1}(\langle\mathscr T u_{2}, \partial_{t}u_{2}\rangle_{\Gamma_{2}}+\langle\mathscr T u_{1}, \partial_{t}u_{2}\rangle_{\Gamma_{2}}){\rm d}t+2{\rm Re} \int_0^t\langle g,\partial_{t}u_{2}\rangle_{\Gamma_{2}}{\rm d}t\nonumber \\
 \leq& 2{\rm Re} \int_0^t(\|g\|_{H^{-1/2}(\Gamma_{1})}\|\partial_{t}u_{1}\|_{H^{1/2}(\Gamma_{1})}+\|g\|_{H^{-1/2}(\Gamma_{2})}\|\partial_{t}u_{2}\|_{H^{1/2}(\Gamma_{2})}){\rm d}t\nonumber \\
 \lesssim&2{\rm Re}\int_0^t(\|g\|_{H^{-1/2}(\Gamma)}\|\partial_{t}u\|_{H^{1}(\Omega)}){\rm d}t \nonumber \nonumber  \\
 \leq&2(\max\limits_{t\in[t,T]}\|\partial_{t}u\|_{H^{1}(\Omega)})\|g\|_{L^{1}(0, T; H^{-1/2}(\Gamma))} .
\end{align}
In order to give the estimate of $\|\partial_t (\nabla u)\|_{L^2 (\Omega)^2}$,
taking the derivative of \eqref{au2} with respect to $t$. We find that $\partial_{t}u$ also satisfies the same equations with $g$ replaced by $\partial_{t}g$.
Hence consider
\begin{align}\label{E2}
e_4 (t)= \|\varepsilon^{1/2} \partial_t^{2} u (\cdot, t)\|^2_{L^2 (\Omega)} +\|\mu^{-1/2} \partial_{t}(\nabla u (\cdot, t))\|^2_{L^2 (\Omega)^2},
\end{align}
and get the estimate
\begin{align}\label{E4}
e_4 (t)
\leq&2{\rm Re} (\int_0^t\int_{\Gamma_{1}}( \partial_{t}g) \partial^{2}_{t}\bar{u}_{1} {\rm d}\gamma_{1}{\rm d}t+\int_0^t\int_{\Gamma_{2}} (\partial_{t}g) \partial^{2}_{t}\bar{u}_{2} {\rm d}\gamma_{2}{\rm d}t) \nonumber \\
=&2{\rm Re} \int_{\Gamma_{1}}(\partial_{t}g)\partial_{t}\bar{u}_{1}|^{t}_{0}{\rm d}\gamma_{1}-2{\rm Re} \int_0^t\int_{\Gamma_{1}}(\partial^{2}_{t}g) \partial_{t}\bar{u}_{1} {\rm d}\gamma_{1}{\rm d}t\nonumber\\
&+2{\rm Re} \int_{\Gamma_{2}}(\partial_{t}g)\partial_{t}\bar{u}_{2}|^{t}_{0}{\rm d}\gamma_{2}-2{\rm Re} \int_0^t\int_{\Gamma_{2}}(\partial^{2}_{t}g) \partial_{t}\bar{u}_{2} {\rm d}\gamma_{2}{\rm d}t \nonumber \\
\leq&2(\max\limits_{t\in[t,T]}\|\partial_{t}u\|_{H^{1}(\Omega)})(\max\limits_{t\in[t,T]}\|\partial_{t}g\|_{H^{-1/2}(\Gamma)}+\|\partial^{2}_{t}g\|_{L^{1}(0, T; H^{-1/2}(\Gamma))}).
\end{align}
Combing the above estimates \eqref{E1}--\eqref{E4}, we can obtain
\begin{align*}
\max\limits_{t\in[t,T]}\|\partial_{t}u\|^{2}_{H^{1}(\Omega)}\lesssim\left(\|g\|_{L^{1}(0, T; H^{-1/2}(\Gamma))}+\max\limits_{t\in[t,T]}\|\partial_{t}g\|_{H^{-1/2}(\Gamma)}+\|\partial^{2}_{t}g\|_{L^{1}(0, T; H^{-1/2}(\Gamma))}\right)\|\partial_{t}u\|_{H^{1}(\Omega)},
\end{align*}
which give the estimate \eqref{aaz} after applying Young's inequality.
\end{proof}
\subsection{A priori estimates of the two cavity problem}

In this section, for the two cavity scattering problem, we derive a priori estimates for the total field with a minimum regularity requirement for the data and an explicit dependence on the time.

The variation problem of \eqref{aab} in time-domain is to find $u_{j}\in H^{1}_{\rm S_{j}}(\Omega_{j}), j=1,2$ for all $t>0$ such that for all $v_j \in  H^{1}_{\rm S_{j}}(\Omega_{j})$
\begin{align}\label{abf}
\int_{\Omega_{j}}\varepsilon_{j}(\partial^{2}_{t}u_{j})\bar{v}_{j}{\rm d}\boldsymbol\rho=-\int_{\Omega_{j}}\mu^{-1}\nabla u_{j}\cdot\nabla\bar{v}_{j}{\rm d}\boldsymbol\rho+\overset{2}{\underset{i=1}{\sum}}\int_{\Gamma_{j}}\mu^{-1}(\mathscr{T}u_{i})\bar{v}_{j}{\rm d}\gamma_{j}+\int_{\Gamma_{j}} g\bar{v}_{j}{\rm d}\gamma_{j}.
\end{align}
This is equivalent to find $u\in H^{1}_{\rm S}(\Omega)$ with $u|_{\Omega_{j}}=u_{j}\in H^{1}_{\rm S_j}(\Omega_{j})$ , such that for all $v\in H^{1}_{\rm S}(\Omega)$ with  $v_{j}=v|_{\Omega_{j}}\in H^{1}_{\rm S_j}(\Omega_{j})$, it holds
\begin{align*}
c_{1}(u, v)=\langle g, v_{1}\rangle_{\Gamma_{1}}+\langle g, v_{2}\rangle_{\Gamma_{2}},
\end{align*}
where the sesquilinear form
\begin{align*}
c_{1}(u, v)&=\overset{2}{\underset{j=1}{\sum}}\left(\int_{\Omega_{j}}\varepsilon_{j}(\partial^{2}_{t}u_{j})\bar{v}_{j}{\rm d}\boldsymbol\rho+\int_{\Omega_{j}}\mu^{-1}\nabla u_{j}\cdot\nabla\bar{v}_{j}{\rm d}\boldsymbol\rho\right)-\overset{2}{\underset{j=1}{\sum}}\overset{2}{\underset{i=1}{\sum}}\int_{\Gamma_{j}}\mu^{-1}(\mathscr{T}u_{i})\bar{v}_{j}{\rm d}\gamma_{j}.
\end{align*}

\begin{theo}\label{abj}
Let $u\in H^{1}_{\rm S}(\Omega)$ be the solution of \eqref{aar}. Given $g\in L^{1}(0, T; H^{-1/2}(\Gamma))$, we have for any $T>0$ that
\begin{align}\label{abk}
  \|u\| _{L^{\infty} (0, T; L^2 (\Omega))}+\|\nabla u\| _{L^{\infty} (0, T; L^2 (\Omega)^2)}
  &\lesssim T \|g\|_{L^1 (0,T; H^{-1/2}(\Gamma))} +\|\partial_{t}g\|_{L^1 (0, T; H^{-1/2}(\Gamma))},
  \end{align}
  and
  \begin{align}\label{abl}
 \|u\|_{L^{2} (0, T; L^2 (\Omega))}+\|\nabla u\| _{L^{2} (0, T; L^2 (\Omega)^2)}
 &\lesssim T^{3/2}\|g\|_{L^1 (0,T; H^{-1/2}(\Gamma))} +T^{1/2} \|\partial_{t}g\|_{L^1 (0, T; H^{-1/2}(\Gamma))}.
\end{align}
\end{theo}
\begin{proof}
Define the test function  $\psi_1$ as in the proof of Theorem \ref{ES4}. Denote by  $\psi_{1}^{(1)}:=\psi_{1}|_{\Omega_{1}}$ and $\psi_{1}^{(2)}:=\psi_{1}|_{\Omega_{2}}.$
Taking  the test functions $v_j=\psi_1^{(j)}$ in \eqref{abf}, we can obtain
\begin{align}\label{abm}
 \overset{2}{\underset{j=1}{\sum}}\int_{\Omega_{j}}\varepsilon_{j}(\partial^{2}_{t}u_{j})\bar{\psi}_{1}^{(j)}{\rm d}\boldsymbol\rho
 =&-\overset{2}{\underset{j=1}{\sum}}\int_{\Omega_{j}}\mu^{-1}\nabla u_{j}\cdot\nabla\bar{\psi}_{1}^{(j)}{\rm d}\boldsymbol\rho
 &+\overset{2}{\underset{j=1}{\sum}}\overset{2}{\underset{i=1}{\sum}}\int_{\Gamma_{j}}\mu^{-1}(\mathscr{T}u_{i})\bar{\psi}_{1}^{(j)}{\rm d}\gamma_{j}
 +\overset{2}{\underset{j=1}{\sum}}\int_{\Gamma_{j}} g\bar{\psi}_{1}^{(j)}{\rm d}\gamma_{j}.
\end{align}
It follows from the facts in \eqref{F1}  and the initial conditions  in \eqref{aab} that
\begin{align*}
 {\rm Re}\overset{2}{\underset{j=1}{\sum}}\int_0^{\theta} \int_{\Omega_{j}}\varepsilon_{j}(\partial^{2}_{t}u_{1})\bar{\psi}_{1}^{(j)}{\rm d}\boldsymbol\rho{\rm d} t
 ={\rm Re}\overset{2}{\underset{j=1}{\sum}} \int_{\Omega_{j}} \varepsilon_{j} \left(  (\partial_t u_{j}) \bar {\psi}_{1}^{(j)} \big|_{0}^{\theta} +\frac{1}{2} |u_{j}|^2 \big|_{0} ^{\theta}\right) {\rm d} \boldsymbol  \rho
 = \frac{1}{ 2} \overset{2}{\underset{j=1}{\sum}}\| \varepsilon^{1/2}_{j}  u_{j} (\cdot, \theta)\|^2_{L^2 (\Omega_{j})}.
\end{align*}
Integrating \eqref{abm} from $t=0$ to $t=\theta$ and taking the real parts yields
\begin{align}
&\overset{2}{\underset{j=1}{\sum}}\left(\frac{1}{ 2} \| \varepsilon^{1/2}_{j}  u_{j} (\cdot, \theta)\|^2_{L^2 (\Omega_{j})}
 +{\rm Re} \int_0^{\theta} \int_{\Omega_{j}} \mu^{-1} \nabla u_{j} \cdot \nabla \bar {\psi}_{1}^{(j)} {\rm d} \boldsymbol  \rho  {\rm d}t\right)
\nonumber \\
= &\overset{2}{\underset{j=1}{\sum}}\frac{1}{ 2} \left(\| \varepsilon^{1/2}_{j}  u_{j} (\cdot, \theta)\|^2_{L^2 (\Omega_{j})}
 +\int_{\Omega_{j}} \mu^{-1} \left| \int_0^{\theta} \nabla u_{j} (\cdot, t) {\rm d} t \right|^2{\rm d} \boldsymbol  \rho\right)
 \nonumber\\
= & {\rm Re}\overset{2}{\underset{j=1}{\sum}}\left(\overset{2}{\underset{i=1}{\sum}}\int_0^{\theta} \mu^{-1}\langle \mathscr Tu_{i},~ \psi_{1}^{(j)} \rangle_{\Gamma_{j}}{\rm d} t
 +\int_0^{\theta} \int_{\Gamma_{j}}g \bar \psi_{1}^{(j)}  {\rm d} \gamma_{j}{\rm d} t
 \right).\label {abn}
\end{align}
In the following, we estimate the two terms on the right-hand side of \eqref{abn} separately.
It follows from  Lemma \ref{aah} that
\begin{align}
{\rm Re}\bigg(\int_0^{\theta} \mu^{-1} \langle \mathscr Tu_{1},~ \psi_{1}^{(1)} \rangle_{\Gamma_{1}}
&+\mu^{-1}\langle \mathscr Tu_{2},~ \psi_{1}^{(1)} \rangle_{\Gamma_{1}} {\rm d }t \nonumber\\
&+\int_0^{\theta} \mu^{-1}\langle \mathscr Tu_{2},~ \psi_{1}^{(2)} \rangle_{\Gamma_{2}}+\mu^{-1} \langle \mathscr Tu_{1},~ \psi_{1}^{(2)} \rangle_{\Gamma_{2}} {\rm d }t\bigg) \leq 0.\label{abo}
\end{align}
For $0 \leq t \leq \theta \leq T,$  by the fact in \eqref{F2}, we have
\begin{align}
 {\rm Re}\overset{2}{\underset{j=1}{\sum}}\int_0^{\theta} \int_{\Gamma_{j}}g \bar \psi_{1}^{(j)}  {\rm d} \gamma_{j}{\rm d} t
 &\leq \overset{2}{\underset{j=1}{\sum}} \left( \int_0^\theta \|g (\cdot, t)\|_{H^{-1/2}(\Gamma_{j})} {\rm d} t\right) \left( \int_0^\theta \|u_{j}(\cdot, t)\|_{H^{1/2}(\Gamma_{j})} {\rm d} t\right)\nonumber\\
 &\leq\left( \int_0^\theta \|g (\cdot, t)\|_{H^{-1/2}(\Gamma)} {\rm d} t\right) \left( \int_0^\theta \|u(\cdot, t)\|_{H^{1/2}(\Gamma)} {\rm d} t\right). \label{abp}
\end{align}
Combining \eqref{abo}--\eqref{abp}, we have for any $\theta \in [0, T]$ that
 \begin{align}
 \frac{1}{2} \overset{2}{\underset{j=1}{\sum}}\| \varepsilon^{1/2}_{j}  u_{j} (\cdot, \theta)\|^2_{L^2 (\Omega_{j})}
 \leq \left( \int_0^\theta \|g(\cdot, t)\|_{L^2 (\Omega)} {\rm d} t\right)
 \left( \int_0^\theta \|u(\cdot, t)\|_{L^2 (\Omega)} {\rm d} t\right). \label{abq}
 \end{align}

Taking the derivative of \eqref{aar} with respect to $t$, we know that $\partial_{t}u$ satisfies the same equation with $g$ replaced by $\partial_t g$.
In similar way,  define
\begin{align*}
  \psi_{2} (\boldsymbol \rho, t) =\int_t^{\theta} \partial_t u_{i}(\boldsymbol \rho, \tau) {\rm d} \tau, \quad \boldsymbol \rho  \in \Omega, ~~~ 0 \leq t \leq \theta,~~~i=1,2,
 \end{align*}
 and denote by $\psi_{2}^{(1)}=\psi_{2}|_{\Omega_{1}}$ and $\psi_{2}^{(2)}=\psi_{2}|_{\Omega_{2}}$.
 It follows the  same step as above
\begin{align}
\overset{2}{\underset{j=1}{\sum}}\frac{1}{ 2} &\left(\| \varepsilon^{1/2}_{j} \partial_{t} u_{j} (\cdot, \theta)\|^2_{L^2 (\Omega_{j})}
+\int_{\Omega_{j}} \mu^{-1} \left| \int_0^{\theta} \partial_{t}(\nabla u_{j} (\cdot, t)) {\rm d} t \right|^2{\rm d} \boldsymbol \rho\right) \nonumber\\
 =& {\rm Re}\overset{2}{\underset{j=1}{\sum}}\left(\overset{2}{\underset{i=1}{\sum}}\int_0^{\theta} \mu^{-1} \langle \mathscr T\partial_{t}u_{i},~ \psi_{2}^{(j)} \rangle_{\Gamma_{j}}{\rm d} t
 +\int_0^{\theta} \int_{\Gamma_{j}}g \bar \psi_{2}^{(j)}  {\rm d} \gamma_{j}{\rm d} t
 \right).
\label {aby}
\end{align}
Integrating by parts yields
 \begin{align}
\frac{1}{2} \int_{\Omega_{j}} \mu^{-1} \left| \int_0^{\theta} \partial_{t}(\nabla u_{j} (\cdot, t)) {\rm d} t \right|^2 {\rm d} \boldsymbol \rho
=\frac{1}{2} \|\mu^{-1/2} \nabla u_{j} (\cdot, \theta) \|^2 _{L^{2}(\Omega_{j})^2}, ~~~\quad ~j=1,~2.\label{abu}
\end{align}
The estimate of the  first term on the right-hand side of \eqref{aby} can be discussed similarly  as above,  we only consider the second term.
By the fact in \eqref{F1} and  Lemma \ref{tt10}, we get
\begin{align}
\overset{2}{\underset{j=1}{\sum}}\int_0^{\theta} \int_{\Gamma_{j}}g \bar \psi_{2}^{(j)}  {\rm d} \gamma_{j}{\rm d} t
 &= \overset{2}{\underset{j=1}{\sum}}\left(\int_{\Gamma_{j}}( \int_0^t \partial_{\tau}g (\cdot, \tau) {\rm d} \tau)\bar{u}_{j} (\cdot, t)|^{\theta}_{0} {\rm d} \gamma_{j}-\int_0^{\theta}\int_{\Gamma_{j}}\partial_{t}g (\cdot, t)u_{j}(\cdot, t)  {\rm d} \gamma_{j}{\rm d} t\right) \nonumber\\
 &\lesssim  \int_0^{\theta}\|\partial_{t}g (\cdot, t)\|_{H^{-1/2}(\Gamma)}\|u(\cdot, t)\|_{H^{1}(\Omega)} {\rm d} t.\label{abv}
\end{align}
Substituting \eqref{abu}--\eqref{abv} into \eqref{aby}, we have for any $\theta \in [0, T]$ that
 \begin{align}
 \frac{1}{2}\sum\limits_{j=1}^2 \left(\|\varepsilon_j^{1/2} \partial_{t}u_j (\cdot, \theta)\|^2_{L^2 (\Omega)}+ \|\mu^{-1/2} \nabla u_j (\cdot, \theta) \|^2 _{L^{2}(\Omega)^2}\right)
 &\lesssim  \int_0^{\theta}\|\partial_{t}g (\cdot, t)\|_{H^{-1/2}(\Gamma)}\|u(\cdot, t)\|_{H^{1}(\Omega)} {\rm d} t. \label{abw}
 \end{align}
Combing the estimates \eqref{abq} and \eqref{abw}, we obtain
\begin{align}
\|u(\cdot, \theta)\|^{2}_{L^{2}(\Omega)}+\|\nabla u\|^{2}_{L^{2}(\Omega)^{2}}
 &\lesssim \left(\int^{\theta}_{0}\|g(\cdot, t)\|_{H^{-1/2}(\Gamma)}{\rm d}t \right)\left(\int^{\theta}_{0}\|u(\cdot, t)\|_{H^{1}(\Omega)}{\rm d}t \right)\nonumber\\
 &+\int^{\theta}_{0}\|\partial_{t}g(\cdot, t)\|_{H^{-1/2}(\Gamma)}\|u(\cdot, t)\|_{H^{1}(\Omega)}{\rm d}t.\label{abz}
\end{align}
Taking the $L^{\infty}$-norm with respect to $\theta$ on both sides of \eqref{abz} yields
\begin{align*}
 \|u\|^2_{L^{\infty} (0, T; L^2 (\Omega))}+ \|\nabla u\|^{2}_{L^{\infty}(0, T; L^2 (\Omega)^{2})}
 &\lesssim T \|g\|_{L^1 (0, T; H^{-1/2 } (\Gamma))}\|u\|_{L^{\infty} (0, T; H^1 (\Omega))} \nonumber\\
 &+\|\partial_{t}g\|_{L^1 (0, T; H^{-1/2 } (\Gamma))}\|u\|_{L^{\infty} (0, T; H^1 (\Omega))},
\end{align*}
which gives the estimate \eqref{abk} after applying the Young's inequality.

Integrating \eqref{abz} with respect to $\theta$ from $0$ to $T$ and using the Cauchy--Schwarz inequality, we obtain
\begin{align*}
 \|u\|^2_{L^{2} (0, T; L^2 (\Omega))}+ \|\nabla u\|^{2}_{L^{2}(0, T; L^2 (\Omega)^{2})}
 &\lesssim T^{3/2} \|g\|_{L^1 (0, T; H^{-1/2 } (\Gamma))}\|u\|_{L^{2} (0, T; H^1 (\Omega))} \nonumber\\
 &+T^{1/2}\|\partial_{t}g\|_{L^1 (0, T; H^{-1/2 } (\Gamma))}\|u\|_{L^{2} (0, T; H^1 (\Omega))},
\end{align*}
which implies the estimate \eqref{abl} by using  Young's inequality again.
\end{proof}
\section{multiple cavities scattering problem}\label{FRP}
In this section, we generalize the model problem and techniques to the case of multiple cavity scattering.
The proofs and results are analogous to those for the two cavity scattering problem.
For completes, we briefly state the results and give the results.
\subsection {Problem formulation}
\begin{figure}
\centering
\includegraphics[width=0.45\textwidth]{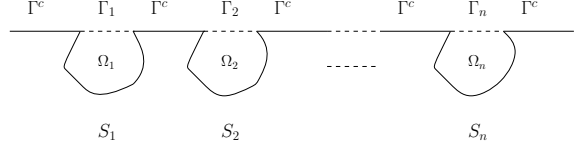}
\caption{The problem geometry of the multiple cavity.}
\label{fig3}
\end{figure}

As shown in the Figure \ref{fig3},  the $n$-multiple open cavities $\Omega_{1}, \Omega_{2},\cdots, \Omega_{n}$ are placed on a perfectly conducting ground plane $\Gamma^{\rm c}$,
 with  apertures $\Gamma_{1}, \Gamma_{2}, \cdots, \Gamma_{n}$ and  walls $S_{1}, S_{2}, \cdots, S_{n}$. Above the flat surface $\{y=0\}=\Gamma^{\rm c}\cup \Gamma_{1}\cup\Gamma_{2}\cup\cdots\cup\Gamma_{n}$,
the medium is assumed to be homogeneous with the positive dielectric permittivity $\varepsilon_{0}$ and magnetic permeability $\mu_{0}$.
The medium inside the cavity $\Omega_{i}$ is inhomogenous with the variable dielectric permittivity $\varepsilon_{j}(x, y)$ and the same magnetic permeability $\mu(x,y)$. Assume further that $\varepsilon_{j}(x, y), \mu(x, y)\in L^{\infty}(\Omega_{j})$
are positive for $j=1, 2,\cdots, n$, and satisfy
\begin{align*}
0<\varepsilon_{j, \min}\leq\varepsilon_{j}\leq\varepsilon_{j, \max}<\infty, \quad 0<\mu_{\min}\leq\mu\leq\mu_{\max}<\infty.
\end{align*}
Consider the similar model of the wave equation for the total field:
\begin{align}\label{md}
 \begin{cases}
  \varepsilon \partial_t^2 u - \nabla \cdot \left(  \mu^{-1} \nabla u\right)=0 \quad  &\text {in} ~~~\Omega^{\rm e}\cup\Omega_1\cup\cdots\cup \Omega_n,~~~t>0,\\
  u \big|_{t=0}=0, \quad \partial_t u \big|_{t=0} =0\quad & \text {in}~~~\Omega^{\rm e}\cup\Omega_1\cup\cdots\cup \Omega_n,\\
  u=0\quad &\text{on} ~~ \Gamma^{\rm c}\cup S_1\cup \cdots  \cup S_n  , ~~t>0.
\end{cases}
\end{align}
The total field $u$ is assumed to consist of the incident field $u^{\rm inc}$, the reflected field $u^{\rm r}$, and the scattered field $u^{\rm sc}$, where the
scattered field is required to satisfy the radiation condition \eqref{SR}.

\subsubsection{Transparent boundary condition}
 As the two cavity situation,
to reduce the scattering problem from the open domain $\Omega^{\rm e}\cup\Omega_{1}\cup\cdot\cdot\cdot\cup\Omega_{n}$
into the bounded domain, we need to derive transparent boundary conditions on the aperture $\Gamma_{j}, ~j=1,...n$ . Reformulating the multiple cavity scattering
problem \eqref{md} into $n$ single cavity scattering problem with the coupled boundary conditions
\begin{align}\label{aca}
 \begin{cases}
  \varepsilon_{j} \partial_t^2 u_{j} - \nabla \cdot \left(  \mu^{-1}_{j} \nabla u_{j}\right)=0 \quad  &\text {in} ~~~\Omega_{j},~~~t>0,\\
  u_{j} \big|_{t=0}=0, \quad \partial_t u_{j} \big|_{t=0} =0\quad & \text {in}~~~\Omega_{j},\\
  u_{j}=0  \quad & \text {on}~~~S_{j}, ~~~~t>0,\\
  \partial_{\boldsymbol n} u_{j} = \mathscr T u_{j}+g \quad & \text {on}~~~\Gamma_{j},~~~~t>0,
\end{cases}
\end{align}
where the transparent boundary operator  $\mathscr T$ will be given later and $u_{j}=u|_{\Omega_{j}},~j=1,...,n$.

Due to the homogeneous medium in the upper half space $\mathbb R^{2}_{+}$ and the radiation condition \eqref{SR}, the scattered field $u^{\rm sc}$ still satisfies the same ordinary differential equation \eqref{abc} after taking the Laplace transform with respect to $t$.
Thus the total field $\breve{u}$ and $u$  satisfy the   transparent boundary conditions in frequency domain and time-domain, respectively:
\begin{align}\label{acd}
\partial_{\mathbf{n}}\breve{u}=\mathscr B \breve{u} +\breve{g},\quad  \partial_{\mathbf{n}}u=\mathscr T u +g  ~\quad &\text {on}~~~\Gamma^{\rm c}\cup\Gamma_{1}\cup\cdot\cdot\cdot\cup\Gamma_{n}.
\end{align}
Next, we derive the transparent boundary condition for each $u_j$ on $\Gamma_j$.

For  $u_{j}(x,0), j=1,..., n$ defined on $\Gamma_j$, we extend them to the  whole $x$-axis by
\begin{align*}
\tilde{u}_{j}(x,0)=
 \begin{cases}
  u_{j}(x,0) ~~\quad & {\rm for} ~x\in \Gamma_{j},\\
  0          ~~\quad & {\rm for} ~x\in \mathbb R\backslash\Gamma_{j}.
\end{cases}
\end{align*}
For the total field $u(x,0)$, define its extension to the whole $x$-axis by
\begin{align*}
\tilde{u}(x,0)=
 \begin{cases}
  u_{j}(x,0) ~~\quad & {\rm for} ~x\in \Gamma_{j},\\
  0          ~~\quad & {\rm for} ~x\in \Gamma^{\rm c}.
\end{cases}
\end{align*}
By the definitions above, it is obvious that
\begin{align*}
\tilde{u}=\overset{n}{\underset{j=1}{\sum}}\tilde{u}_{j} ~~\quad {\rm on}~ \Gamma^{\rm c}\cup\Gamma_{1}\cup\cdot\cdot\cdot\cup\Gamma_{n}.
\end{align*}
The  transparent boundary condition \eqref{acd} can be written as
\begin{align*}
\partial_{\mathbf{n}}{\breve{ \tilde u}}=\mathscr B {\breve{ \tilde u}} +\breve{g}  ~\quad\text {on}~~~\Gamma^{\rm c}\cup\Gamma_{1}\cup\cdot\cdot\cdot\cup\Gamma_{n};\quad\partial_{\mathbf{n}}\tilde{u}=\mathscr T \tilde{u} +g ~\quad\text {on}~~~\Gamma^{\rm c}\cup\Gamma_{1}\cup\cdot\cdot\cdot\cup\Gamma_{n},
\end{align*}
which leads to the transparent boundary conditions for $u_{j}$ on $\Gamma_j$:
\begin{align}\label{acg}
\partial_{\mathbf{n}}\breve{u}_{j}=\mathscr B \breve{\tilde{u}}_{j}+\overset{n}{\underset{\underset{i\neq j}{i=1}}{\sum}}\mathscr B \breve{\tilde{u}}_{i}+\breve{g} ~~\quad {\rm on}~ \Gamma_{j};\quad\partial_{\mathbf{n}}u_{j}=\mathscr T \tilde{u}_{j}+\overset{n}{\underset{\underset{i\neq j}{i=1}}{\sum}}\mathscr T \tilde{u}_{i}+g ~~\quad {\rm on}~ \Gamma_{j},t>0.
\end{align}
From \eqref{acg}, it is obvious that the boundary conditions for $u_{1},...,u_{n}$ are coupled with each other,
which is the major difference between the single cavity scattering problem and the multiple cavity scattering problem.

The following lemma is analogous to  Lemma \ref{aag}, which  is  used  for analysis the uniqueness and existence for  the multiple cavity scattering problem.
\begin{lemm}\label{ach}
 It holds that
 \begin{align*}
  -{\rm Re}\overset{n}{\underset{j=1}{\sum}} \overset{n}{\underset{i=1}{\sum}}\langle (s\mu)^{-1} \mathscr B u_{i}, u_{j} \rangle_{\Gamma_{j}}£¬ \geq 0, \quad u_{j}\in H^{1/2} (\mathbb R), ~j=1,2...,n.
 \end{align*}

\end{lemm}
\begin{proof}
By definition \eqref{acg}, recalling $\beta^{2}(\xi)=\xi^{2}+c^{-2}s^{2}$,  it gives
\begin{align*}
  {\rm Re}\left(\overset{n}{\underset{j=1}{\sum}} \overset{n}{\underset{i=1}{\sum}} \langle (s\mu)^{-1} \mathscr B u_{i}, u_{j} \rangle_{\Gamma_{j}}\right)
  =&\overset{n}{\underset{j=1}{\sum}}\overset{n}{\underset{i=1}{\sum}} \int_{\mathbb R}\frac{1}{\mu|s|^{2}}\frac{s_{1}}{\varsigma}(\zeta^{2}+c^{2}s_{2}^{2}) u_{i}\bar{u}_{j}{\rm d}\xi\\
  \leq&\overset{n}{\underset{j=1}{\sum}}\overset{n}{\underset{i=1}{\sum}} \int_{\mathbb R}\frac{1}{\mu|s|^{2}}\frac{s_{1}}{\varsigma}(\zeta^{2}+c^{2}s_{2}^{2}) u_{i}\bar{u}_{j}{\rm d}\xi\\
  \lesssim&\int_{\mathbb R}\frac{1}{\mu|s|^{2}}\frac{s_{1}}{\varsigma}(\zeta^{2}+c^{2}s_{2}^{2}) \left|\overset{n}{\underset{j=1}{\sum}}u_{j}\right|^{2}{\rm d}\xi\leq0.
\end{align*}
\end{proof}
\begin{lemm}\label{aci}
 For any $u_{j}(\cdot , t)\in L^2 (0, T; H^{1/2} (\Omega_j) )$ with initial value $u_j(\cdot, 0)=0$, denote their zero extension on $L^2 (0, T; H^{1/2} (\mathbb R))$ by $\tilde u_{j} (\cdot, t), j=1,\cdots,n$. Then,
  it holds
 \begin{align*}
  -{\rm Re}\int_0^T \overset{n}{\underset{j=1}{\sum}} \overset{n}{\underset{i=1}{\sum}}\langle \mathscr T \tilde  u_{i} , \partial_t \tilde u_{j}  \rangle_{\Gamma_{j}} {\rm d} t \geq 0.
 \end{align*}

\end{lemm}
\begin{proof}
Extending $\tilde u_j (\cdot, t)$  with respect to $t$ in $\mathbb R$ such that
 $\tilde u_j (\cdot, t)= 0$ outside the interval $[0, T]$, for convenience, we still
  denote it by $\tilde u_j (\cdot, t)$.
Let $\breve {\tilde u}_j =\mathscr L (\tilde u_j)$ be the Laplace of $\tilde u_j$. By the Parseval identity \eqref{PI} and Lemma \ref{ach}, we get
 \begin{align*}
  -{\rm Re} \int_0^{T} e^{- 2 s_1 t} \overset{n}{\underset{j=1}{\sum}} \overset{n}{\underset{i=1}{\sum}} \langle \mathscr T u_{i} , \partial_t u_{j}  \rangle_{\Gamma_{j}} {\rm d} t
  =& - {\rm Re}  \overset{n}{\underset{j=1}{\sum}} \overset{n}{\underset{i=1}{\sum}} \int_{\Gamma_{j}} \int_0^{\infty} e^{- 2 s_1 t} (\mathscr T \tilde u_{i}) \partial_t \bar {\tilde u}_{j} {\rm d} t{\rm d} \gamma\\
  =&-\frac{1}{2 \pi} \int_{\infty}^{\infty}  {\rm Re}\overset{n}{\underset{j=1}{\sum}} \overset{n}{\underset{i=1}{\sum}} \langle \mathscr B \breve {\tilde u}_{i}, s \breve{\tilde u}_{j} \rangle_{\Gamma_{j}}{\rm d} s_2 \\
  =&-\frac {1}{2 \pi} \int_{- \infty}^{\infty} |s|^2 {\rm Re}\overset{n}{\underset{j=1}{\sum}} \overset{n}{\underset{i=1}{\sum}} \int_{\mathbb{R}}s^{-1}\beta(\xi)\breve {\tilde u}_{i}\bar{\breve {\tilde u}}_{j}{\rm d}\xi{\rm d} s_2\\
  =&-\frac {1}{2 \pi} \int_{- \infty}^{\infty} \int_{\mathbb{R}}\frac{s_{1}}{\zeta}(\zeta^{2}+c^{-2}s_{2}^{2})|\overset{n}{\underset{j=1}{\sum}}\breve {\tilde u}_{j}|^{2}{\rm d}\xi{\rm d} s_2 \geq 0,
 \end{align*}
which  completes the proof
after taking $s_1 \rightarrow 0 $.
\end{proof}

\subsection{The reduced multiple cavity scattering problem}
We now present the  well-posedness and stability  of the reduced problem.
For simplicity, we shall use the same notation as those
adopted in Section \ref{TSP} for the two cavity scattering problem.
Denote by $\Omega=\Omega_{1}\cup\cdots\cup\Omega_{n}, \Gamma=\Gamma_{1}\cup\cdots\cup\Gamma_{n}$, and $S=S_{1}\cup\cdots\cup S_{n}$.
Define the  trace functional space
\begin{align*}
\tilde{H}^{1/2}(\Gamma)=\tilde{H}^{1/2}(\Gamma_{1})\times\cdots\times\tilde{H}^{1/2}(\Gamma_{n}),
\end{align*}
whose norm is characterized by
$
\|u\|^{2}_{\tilde{H}^{1/2}(\Gamma)}=\overset{n}{\underset{j=1}{\sum}}\|u_{j}\|^{2}_{\tilde{H}^{1/2}(\Gamma_{j})} .
$
Denote  by
\[H^{-1/2}(\Gamma)=H^{-1/2}(\Gamma_{1})\times\cdots\times H^{-1/2}(\Gamma_{n}),\]
 which is the dual space of $\tilde{H}^{1/2}(\Gamma)$. The norm on the space $H^{-1/2}(\Gamma)$ is characterized by
\begin{align*}
\|u\|^{2}_{H^{-1/2}(\Gamma)}=\overset{n}{\underset{j=1}{\sum}}\|u_{j}\|^{2}_{H^{-1/2}(\Gamma_{j})}.
\end{align*}
Denote the space
\begin{align*}
H^{1}_{\rm S}(\Omega)=H^{1}_{\rm S_{1}}(\Omega_{1})\times\cdots\times H^{1}_{\rm S_{n}}(\Omega_{n}),
\end{align*}
which is a Hilbert space with norm characterized by
$
\|u\|^{2}_{H_{\rm S}^{1}(\Omega)}=\overset{n}{\underset{j=1}{\sum}}\|u_{j}\|^{2}_{H_{\rm S_j}^{1}(\Omega_{j})}.
$
\subsubsection{well-posedness in the $s$-domain}
Taking the Laplace transform of \eqref{aca}, we can get for $j=1, \cdots, n,$
\begin{align}\label{sdm}
\begin{cases}
\varepsilon_j s \breve u_j-\nabla\cdot \left(s^{-1} \mu^{-1} \nabla \breve u_j \right)=0 \quad &\text{in}~~\Omega_j,\\
\breve u_j=0 \quad &\text{on}~~S_j,\\
\partial_{\boldsymbol n} \breve u_j=\mathscr B \breve u_j+\breve g \quad &\text{on}~~\Gamma_j.
\end{cases}
\end{align}
 Multiplying the complex conjugate of test function $v_{j}\in H^{1}_{\rm S_{j}}(\Omega_{j})$ on both sides of the first equality of the \eqref{sdm}
  and integrating over $\Omega_{j}$, we have
\begin{align*}
\int_{\Omega_{j}}(s\mu)^{-1}\nabla\breve{u}_{j}\nabla\bar{v}_{j}+s\varepsilon_{j}\breve{u}_{j}\bar{v}_{j}{\rm d} \boldsymbol \rho-\overset{n}{\underset{i=1}{\sum}}\langle(s\mu)^{-1}\mathscr B \tilde{\breve{u}}_{i}, \tilde{v}_{j}\rangle_{\Gamma_{j}}=\langle\breve{g}, v_{j}\rangle_{\Gamma_{j}}.
\end{align*}
The variational formulation for the multiple cavity scattering problem \eqref{sdm}: find $ \breve u\in H^{1}_{\rm S}(\Omega)$ with  $\breve u|_{\Omega_{j}}=u_{j}\in H^{1}_{\rm S_j}(\Omega_{j})$,
such that for all $v\in H^{1}_{\rm S}(\Omega)$ with $v|_{\Omega_{j}}=v_{j}\in H^{1}_{\rm S_j}(\Omega_{j})$, it holds
\begin{align}\label{ack}
a_3(\breve{u}, v)=\overset{n}{\underset{j=1}{\sum}}\langle\breve{g}, v_{j}\rangle_{\Gamma_{j}},
\end{align}
where the sesquilinear form
\begin{align*}
a_3(\breve{u}, v)&=\overset{n}{\underset{j=1}{\sum}}\int_{\Omega_{j}}((s\mu)^{-1}\nabla\breve{u}_{j}\nabla\bar{v}_{j}+s\varepsilon_{j}\breve{u}_{j}\bar{v}_{j}){\rm d} \boldsymbol \rho-\overset{n}{\underset{j=1}{\sum}}\overset{n}{\underset{i=1}{\sum}}\langle(s\mu)^{-1}\mathscr B \tilde{\breve{u}}_{i}, \tilde{v}_{j}\rangle_{\Gamma_{j}}.
\end{align*}

\begin{theo}\label{acm}
 The variational problem \eqref{ack} has a unique solution $\breve{u} \in H^1_{\rm S} (\Omega)$ which satisfies
 \begin{align}\label{acn}
  \|\nabla \breve{u}\|_{L^2 (\Omega)^2} +\|s \breve{u}\|_{L^2 (\Omega)} \lesssim s_1^{-1} \| s\breve{g}\|_{H^{-1/2} (\Gamma)}.
 \end{align}
\end{theo}
\begin{proof}
The continuity of the  sesquilinear form $a_3(\breve u, v)$ follows
 \begin{align*}
  |a_{3} (\breve{u}, v)|
  \leq &\frac{1}{|s| \mu}( \overset{n}{\underset{j=1}{\sum}}\|\nabla \breve{u}_{j}\|_{L^2(\Omega_{j})^2} \|\nabla v_{j}\|_{L^2 (\Omega_{j})^2})
  +|s|\varepsilon_{\max} (\overset{n}{\underset{j=1}{\sum}}\|\breve{u}_{j}\|_{L^2 (\Omega_{j})}\|v_{j}\|_{L^2 (\Omega_{j})})\\
  &+\frac{1}{|s|\mu}(\overset{n}{\underset{j=1}{\sum}}\overset{n}{\underset {i=1}{\sum}}\|\mathscr B \breve{u}_{i}\|_{H^{-1/2} (\Gamma_{i})} \|v_{j}\|_{H^{1/2}(\Gamma_{j})})\\
  \lesssim  & \|\nabla \breve{u}\|_{L^2(\Omega)^2} \|\nabla v\|_{L^2 (\Omega)^2}
  +\|\breve{u}\|_{L^2 (\Omega)}\|v\|_{L^2 (\Omega)}+  \|\breve{u}\|_{H^{1/2} (\Gamma)} \|v\|_{H^{1/2} (\Gamma)}\\
  \lesssim  &\|\breve{u}\|_{H^1(\Omega)} \|v\|_{H^1(\Omega)},
 \end{align*}
 where $\varepsilon_{\max}=\max\left\{ \varepsilon_{j}, ~j=1,...,n \right\}$.
A simple calculation yields
\begin{align}\label{aco}
 a_{3} (\breve{u}, \breve{u}) =\overset{n}{\underset{j=1}{\sum}}\int_{\Omega_{j}}(s\mu)^{-1}|\nabla\breve{u}_{j}|^{2}+s\varepsilon_{j}|\breve{u}_{j}|^{2}
{\rm d} \boldsymbol \rho-\overset{n}{\underset{j=1}{\sum}}\overset{n}{\underset{i=1}{\sum}}\langle(s\mu)^{-1}\mathscr B \tilde{\breve{u}}_{i}, \tilde{\breve{u}}_{j}\rangle_{\Gamma_{j}}.
\end{align}
Taking the real part of \eqref{aco} and using Lemma \ref{ach}, we get
\begin{align}\label{acp}
 {\rm Re} \left(  a_{3} (\breve{u}, \breve{u}) \right) \geq C_{1} \frac{s_1}{|s|^2} \left( \|\nabla \breve{u}\|^2_{L^2 (\Omega)^2}  +\|s \breve{u}\|^2_{L^2 (\Omega)}\right),
\end{align}
where $C_{1}=\min\{\mu^{-1},1\}$.
It follows from the Lax--Milgram lemma that the variational problem \eqref{ack} has a unique solution $\breve{u} \in H^1_{\rm S} (\Omega).$
Moreover, we have from \eqref{ack} that
\begin{align}\label{acq}
 |a_{3}(\breve{u}, \breve{u})| \leq |s|^{-1} \|\breve{g}\|_{H^{-1/2} (\Gamma)} \|s \breve{u}\|_{L^2 (\Omega)}.
\end{align}
Combining \eqref{acp}--\eqref{acq} leads to
\begin{align*}
 \|\nabla \breve{u}\|^2_{L^2 (\Omega)^2}  +\|s \breve{u}\|^2_{L^2 (\Omega)} \lesssim s_1^{-1}  \| s \breve{g}\|_{H^{-1/2} (\Gamma)} \|s \breve{u}\|_{L^2 (\Omega)},
\end{align*}
which implies the   estimate  of \eqref{acn} after applying the Young's inequality.

\end{proof}

\subsubsection{well-posedness in the time-domain}
Using the time-domain transparent boundary condition, we consider the reduced initial-boundary value problem:
\begin{align}\label{acr}
\begin{cases}
  \varepsilon \partial_t^2 u - \nabla \cdot \left(  \mu^{-1} \nabla u\right)=0 \quad  &\text {in} ~~~\Omega,~~~t>0,\\
   u\big|_{t=0}=0, \quad \partial_t u \big|_{t=0} =0\quad & \text {in}~~~\Omega,\\
  u=0  \quad & \text {on}~~~S, ~~~~t>0,\\
  \partial_r u = \mathscr T u+g \quad & \text {on}~~~\Gamma,~~~~t>0.
\end{cases}
\end{align}

\begin{theo}\label{acy}
The initial-boundary problem \eqref{acr} has a unique solution $u$, which satisfies
\begin{align*}
u\in L^{2}(0, T; H^{1}_{\rm S}(\Omega))\cap H^{1}(0, T; L^{2}(\Omega)),
\end{align*}
and the stability estimate
\begin{align} \label{acz}
 \begin{split}
 \max\limits_{t\in[t,T]}&\left(\|\partial_t u\|_{L^2 (\Omega)}  +\|\partial_{t}(\nabla u)\|_{L^2 (\Omega)^2}\right) \\
 & \lesssim (\|g\|_{L^{1}(0, T; H^{-1/2}(\Gamma))}+ \max\limits_{t\in[t,T]}\|\partial_{t}g\|_{H^{-1/2}(\Gamma)}+\|\partial^{2}_{t}g\|_{L^{1}(0, T; H^{-1/2}(\Gamma))}).
 \end{split}
\end{align}
\end{theo}
\begin{proof}
Using the same way of the two cavity scattering problem, we can get that
\begin{align*}
  u\in L^2 \left( 0, T; H^1_{\rm S} (\Omega) \right) \cap H^1 \left( 0, T; L^2 (\Omega) \right).
\end{align*}

Next, we prove the stability.
For any $0<t<T$, define the energy function
\begin{align*}
 e_5 (t) =\| \varepsilon^{1/2} \partial_t u (\cdot, t)\|^2_{L^2 (\Omega)} +\|  \mu^{-1/2} \nabla u (\cdot, t)\|^2_{L^2 (\Omega)^2}.
\end{align*}
It follows from \eqref{acr} and integration by parts that
\begin{align*}
\int_0^t e_5'(\tau ) {\rm d} \tau
 &= 2 {\rm Re}\overset{n}{\underset{j=1}{\sum}}\int_0^t \int_{\Omega_{j}} \left(  \nabla \cdot \left(  \mu^{-1} \nabla u_{j}\right) \partial_t \bar u_{j} +\mu^{-1} (\nabla \partial_t u_{j}) \cdot \nabla \bar u_{j} \right) {\rm d} \boldsymbol \rho {\rm d} \tau.
\end{align*}
Since $e_{5}(0)=0$, we obtain from Lemma \ref{aci} that
\begin{align*}
e_5(t)
 &= 2{\rm Re} \overset{n}{\underset{j=1}{\sum}}\overset{n}{\underset{i=1}{\sum}}\int_0^t\mu^{-1}\langle\mathscr T u_{i}, \partial_{t}u_{j}\rangle_{\Gamma_{j}}{\rm d}t+2{\rm Re}\overset{n}{\underset{j=1}{\sum}} \int_0^t\langle g,\partial_{t}u_{j}\rangle_{\Gamma_{j}}{\rm d}t\\
 &\leq 2{\rm Re}\overset{n}{\underset{j=1}{\sum}} \int_0^t\|g\|_{H^{-1/2}(\Gamma_{j})}\|\partial_{t}u_{j}\|_{H^{1/2}(\Gamma_{j})}{\rm d}t\\
 &\lesssim2{\rm Re}\int_0^t\|g\|_{H^{-1/2}(\Gamma)}\|\partial_{t}u\|_{H^{1}(\Omega)}{\rm d}t\\
 &\leq2(\max\limits_{t\in[t,T]}\|\partial_{t}u\|_{H^{1}(\Omega)})\|g\|_{L^{1}(0, T; H^{-1/2}(\Gamma))} .
\end{align*}

Taking the derivative of \eqref{au2} with respect to $t$, we know that $\partial_{t}u$ also satisfies the same equations with $g$ replaced by $\partial_{t}g$. In order to control $\| \partial_{t}(\nabla u (\cdot, t))\|_{L^2 (\Omega)^2}$,
consider the  energy function
\begin{align*}
e_6 (t)= \|\varepsilon^{1/2} \partial_t^{2} u (\cdot, t)\|^2_{L^2 (\Omega)} +\|\mu^{-1/2} \partial_{t}(\nabla u (\cdot, t))\|^2_{L^2 (\Omega)^2}.
\end{align*}
Similarly, we  get the estimate
\begin{align*}
e_6 (t)
&\leq2{\rm Re} \overset{n}{\underset{j=1}{\sum}}\int_0^t\int_{\Gamma_{j}} (\partial_{t}g )\partial^{2}_{t}\bar{u}_{j} {\rm d}\gamma_{j}{\rm d}t\\
&=2{\rm Re} \overset{n}{\underset{j=1}{\sum}} \int_{\Gamma_{j}}(\partial_{t}g)\partial_{t}\bar{u}_{j}|^{t}_{0}{\rm d}\gamma_{j}-2{\rm Re}\overset{n}{\underset{j=1}{\sum}} \int_0^t\int_{\Gamma_{j}}(\partial^{2}_{t}g) \partial_{t}\bar{u}_{j} {\rm d}\gamma_{j}{\rm d}t\\
&\leq2(\max\limits_{t\in[t,T]}\|\partial_{t}u\|_{H^{1}(\Omega)})(\max\limits_{t\in[t,T]}\|\partial_{t}g\|_{H^{-1/2}(\Gamma)}+\|\partial^{2}_{t}g\|_{L^{1}(0, T; H^{-1/2}(\Gamma))}).
\end{align*}
Combing the above estimates, we can obtain
\begin{align*}
\max\limits_{t\in[t,T]}\|\partial_{t}u\|^{2}_{H^{1}(\Omega)}\lesssim(\|g\|_{L^{1}(0, T; H^{-1/2}(\Gamma))}+\max\limits_{t\in[t,T]}\|\partial_{t}g\|_{H^{-1/2}(\Gamma)}+\|\partial^{2}_{t}g\|_{L^{1}(0, T; H^{-1/2}(\Gamma))})\|\partial_{t}u\|_{H^{1}(\Omega)},
\end{align*}
which give the estimate \eqref{acz} after applying the Young's inequality.
\end{proof}
\subsection{A priori estimates of the multiple cavity problem}

In this section, for the multiple cavity scattering problem, we also derive a priori estimates for the total field with a minimum regularity requirement for the data and an explicit dependence on the time.

The variation formulation of \eqref{aca} is to find $u_{j}\in H^{1}_{\rm S_j}(\Omega_{j})$ for all $t>0$ such that
\begin{align*}
\int_{\Omega_{j}}\varepsilon_{j}(\partial^{2}_{t}u_{j})\bar{v}_{j}{\rm d}\boldsymbol \rho=-\int_{\Omega_{j}}\mu^{-1}\nabla u_{j}\cdot\nabla\bar{v}_{j}{\rm d} \boldsymbol \rho+\overset{n}{\underset{i=1}{\sum}}\int_{\Gamma_{j}}\mu^{-1}(\mathscr{T}u_{i})\bar{v}_{j}{\rm d}\gamma_{j}+\int_{\Gamma_{j}} g\bar{v}_{j}{\rm d}\gamma_{j},\quad j=1, \cdots, n.
\end{align*}
This is equivalent to: find $u\in H^{1}_{\rm S}(\Omega)$ with $u|_{\Omega_{j}} =u_{j}\in H^{1}_{\rm S_j}(\Omega_{j})$,
such that for all $v\in H^{1}_{\rm S}(\Omega)$ with  $v_{j}=v|_{\Omega_{j}}\in H^{1}_{\rm S_j}(\Omega_{j})$, it holds
\begin{align*}
c_{2}(u, v)=\overset{n}{\underset{j=1}{\sum}}\langle g, v_{j}\rangle_{\Gamma_{j}},
\end{align*}
where the sesquilinear form
\begin{align*}
c_{2}(u, v)&=\overset{n}{\underset{j=1}{\sum}}\left(\int_{\Omega_{j}}\varepsilon_{j}(\partial^{2}_{t}u_{j})\bar{v}_{j}{\rm d}\boldsymbol \rho+\int_{\Omega_{j}}\mu^{-1}\nabla u_{j}\cdot\nabla\bar{v}_{j}{\rm d}\boldsymbol \rho \right)-\overset{n}{\underset{j=1}{\sum}}\overset{n}{\underset{i=1}{\sum}}\int_{\Gamma_{j}}\mu^{-1}(\mathscr{T}u_{i})\bar{v}_{j}{\rm d}\gamma_{j}.
\end{align*}
\begin{theo}
Let $u\in H^{1}_{\rm S}(\Omega)$ be the solution of \eqref{acr}. Given $g\in L^{1}(0, T; H^{-1/2}(\Gamma))$, we have for any $T>0$ that
\begin{align*}
  \|u\| _{L^{\infty} (0, T; L^2 (\Omega))}+\|\nabla u\| _{L^{\infty} (0, T; L^2 (\Omega)^2)}
  &\lesssim T \|g\|_{L^1 (0,T; H^{-1/2}(\Gamma))} +\|\partial_{t}g\|_{L^1 (0, T; H^{-1/2}(\Gamma))},
  \end{align*}
  and
  \begin{align*}
 \|u\|_{L^{2} (0, T; L^2 (\Omega))}+\|\nabla u\| _{L^{2} (0, T; L^2 (\Omega)^2)}
 &\lesssim T^{3/2}\|g\|_{L^1 (0,T; H^{-1/2}(\Gamma))} +T^{1/2} \|\partial_{t}g\|_{L^1 (0, T; H^{-1/2}(\Gamma))}.
\end{align*}
\end{theo}
The proof is similar in
nature as that of the two cavity model problem and is omitted here for brevity.

\section{Conclusion}\label{CL}
The problem of electromagnetic scattering by cavities embedded in the infinite two-dimensional ground plane is an
important area of research.
In this paper, we present the multiple cavity scattering problem in time-domain. We reduce the  overall scattering problem to coupled  scattering problem in bounded domain
 via the introduction of a novel transparent boundary condition over the cavity aperture in time-domain.
 The  uniqueness, existence  and stability of the reduced problem are obtained in frequency domain and time-domain, respectively.
 The main ingredients of the proofs are the
Laplace transform, the Lax-Milgram lemma, and the Parseval identity. Moreover,
by directly considering the variational problem of the time-domain wave equation,
we obtain a priori estimates with an explicit dependence on the time.


\end{document}